\documentclass[final,onefignum,onetabnum]{siamart190516}

\usepackage{amsmath,amsfonts,amsopn,amssymb}
\usepackage{graphicx,subfig}
\usepackage{multirow,relsize}
\usepackage{url}
\ifpdf
  \DeclareGraphicsExtensions{.eps,.pdf,.png,.jpg}
\else
  \DeclareGraphicsExtensions{.eps}
\fi

\numberwithin{theorem}{section}
\newsiamremark{assumption}{Assumption}
\newsiamremark{remark}{Remark}
\crefname{assumption}{Assumption}{Assumptions}
\crefname{remark}{Remark}{Remarks}

\headers{Additive Schwarz as Gradient Methods}{Jongho Park}

\title{Additive Schwarz Methods for Convex Optimization as Gradient Methods\thanks{Submitted to the editors DATE.
\funding{This research was supported by Basic Science Research Program through the National Research Foundation of Korea~(NRF) funded by the Ministry of Education~(2019R1A6A3A01092549).}
}}

\author{Jongho Park\thanks{Department of Mathematical Sciences, KAIST, Daejeon 34141, Korea 
  (\email{jongho.park@kaist.ac.kr}, \url{https://sites.google.com/view/jonghopark}).}}

\ifpdf
\hypersetup{
  pdftitle={Additive Schwarz Methods for Convex Optimization as Gradient Methods},
  pdfauthor={Jongho Park}
}
\fi

\usepackage{algorithm, algorithmic}

\newcommand\gap{\hspace{0.1cm}}

\newcommand{\bu}{\mathbf{u}}
\newcommand{\bv}{\mathbf{v}}
\newcommand{\bw}{\mathbf{w}}
\newcommand{\bS}{\mathbf{S}}
\newcommand{\T}{\mathcal{T}}
\newcommand{\tF}{\tilde{F}}

\newcommand{\un}{u^{(n)}}
\newcommand{\unn}{u^{(n+1)}}
\newcommand{\Rw}{R_k^* w_k}
\newcommand{\dkn}{d_k^{(n)}}
\newcommand{\dkns}{d_k^{(n)*}}
\newcommand{\M}{M_{\tau, \omega}}
\newcommand{\Mn}{M_{\tau, \omega}^{(n)}}
\newcommand{\kASM}{\kappa_{\textrm{ASM}}}
\newcommand{\intO}{\int_{\Omega}}
\newcommand{\sumk}{\sum_{k=1}^{N}}
\newcommand{\minu}{\min_{u \in V}}
\newcommand{\mint}{\min_{t \in [0,\theta]}}

\let\div\relax
\DeclareMathOperator{\div}{div}

\DeclareMathOperator{\dom}{dom}
\DeclareMathOperator*{\argmin}{\arg\min}
\newcommand{\BigBox}{\vcenter{\hbox{$\mathlarger{\mathlarger{\mathlarger{\mathlarger{\Box}}}}$}}}
\DeclareMathOperator*{\bigsquare}{\BigBox}
\newcommand{\BigBoxInLine}{\vcenter{\hbox{$\mathlarger{\mathlarger{\mathlarger{\Box}}}$}}}
\DeclareMathOperator*{\bigsquareinline}{\BigBoxInLine}

\begin{document}

\maketitle

\begin{abstract}
This paper gives a unified convergence analysis of additive Schwarz methods for general convex optimization problems.
Resembling the fact that additive Schwarz methods for linear problems are preconditioned Richardson methods, we prove that additive Schwarz methods for general convex optimization are in fact gradient methods.
Then an abstract framework for convergence analysis of additive Schwarz methods is proposed.
The proposed framework applied to linear elliptic problems agrees with the classical theory.
We present applications of the proposed framework to various interesting convex optimization problems such as nonlinear elliptic problems, nonsmooth problems, and nonsharp problems.
\end{abstract}

\begin{keywords}
additive Schwarz method, gradient method, convex optimization, convergence analysis, domain decomposition method
\end{keywords}

\begin{AMS}
65N55, 65J15, 65K15, 90C25
\end{AMS}

\section{Introduction}
\label{Sec:Introduction}
Many modern iterative methods such as block relaxation methods, multigrid methods, and domain decomposition methods for linear systems belong to Schwarz methods, also known as subspace correction methods.
Because of this fact, constructing an abstract convergence theory for Schwarz methods has been considered an important task in the field of numerical analysis.
There is a vast literature on the convergence theory of Schwarz methods for linear systems.
The paper~\cite{Xu:1992} by Xu contains an outstanding survey on some early results on Schwarz methods.
Several variants of the convergence theory with various viewpoints were proposed in, e.g.,~\cite{FS:2001,GO:1995,XZ:2002}.
For a modern representation of the abstract convergence theory with historical remarks, one may refer to the monograph~\cite{TW:2005} by Toselli and Widlund.

While the convergence theory of Schwarz methods for linear elliptic problems seems to be almost complete, there has still been much research on convergence analysis of Schwarz methods for nonlinear and nonsmooth problems.
The papers~\cite{TE:1998,TX:2002} are important early results on Schwarz methods for nonlinear problems.
In~\cite{BTW:2003,BW:2000,Tai:2003}, Schwarz methods for variational inequalities which arise in quadratic optimization with constraints were proposed.
Convergence analysis for Schwarz methods was successfully extended to nonquadratic and nonsmooth variational inequalities in~\cite{Badea:2006} and~\cite{BK:2012}, respectively.
Recently, overlapping Schwarz methods for convex optimization problems lacking strong convexity were proposed in~\cite{CTWY:2015,Park:2019}, especially for total variation minimization problems arising in mathematical imaging.
On the other hand, it was shown in~\cite{LN:2017} that Schwarz methods may not converge to a correct solution in the case of nonsmooth convex optimization.

One of the most important observations done in the convergence theory of Schwarz methods for linear problems is that Schwarz methods can be viewed as preconditioned Richardson methods; see, e.g.,~\cite{TW:2005}.
This observation makes the convergence analysis of a method fairly simple; convergence is obvious by the well-known convergence results on the Richardson method and one only need to estimate the condition number of the linear system to obtain an estimate for the convergence rate.
However, such an observation does not exist for general nonlinear and nonsmooth problems.
Due to this situation, all of the aforementioned works on nonlinear problems provided proofs on why their methods converge to a solution correctly with some complex computations.
To the best of our knowledge, the only relevant result on nonlinear problems is~\cite{LP:2019b}; it says that block Jacobi methods for a constrained quadratic optimization problem can be regarded as preconditioned forward-backward splitting algorithms~\cite{BT:2009}.

In this paper, we show that additive Schwarz methods for general convex optimization can be represented as gradient methods.
In the field of mathematical optimization, there has been much research on gradient methods for solving convex optimization problems; for example, see~\cite{BT:2009,CP:2016,Nesterov:2013}.
Therefore, by observing that additive Schwarz methods are interpreted as gradient methods, we can borrow many valuable tools on convergence analysis from the field of mathematical optimization in order to analyze Schwarz methods.
Consequently, we propose a novel abstract convergence theory of additive Schwarz methods for convex optimization.
The proposed framework directly generalizes the classical convergence theory presented in~\cite{TW:2005} for linear elliptic problems to general convex optimization problems.
We also highlight that our framework gives a better convergence rate than existing works~\cite{Badea:2006,BK:2012,TX:2002} for some applications.

Various applications of the proposed convergence theory are presented in this paper.
A very broad range of convex optimization problems fits into our framework.
In particular, we provide examples of nonlinear elliptic problems, nonsmooth problems, and problems without sharpness, where those classes of problems were considered in existing works~\cite{TX:2002}, \cite{BTW:2003,Tai:2003}, \cite{BK:2012}, and~\cite{CTWY:2015,Park:2019}, respectively.

The rest of this paper is organized as follows.
In \cref{Sec:Pre}, we provide some useful tools of convex analysis required in this paper.
An abstract gradient method for solving general convex optimization is introduced in \cref{Sec:Gradient} with the convergence analysis motivated by~\cite{Nesterov:2013}.
In \cref{Sec:ASM}, we show that additive Schwarz methods for convex optimization are indeed gradient methods; a novel abstract convergence theory for additive Schwarz methods is proposed in this viewpoint.
One- and two-level overlapping domain decomposition settings and some important stable decomposition estimates are summarized in \cref{Sec:DD}.
Applications of the proposed convergence theory to various convex optimization problems are presented in \cref{Sec:Applications}.
We conclude this paper with remarks in~\cref{Sec:Conclusion}.

\section{Preliminaries}
\label{Sec:Pre}
In this section, we introduce notation and basic notions of convex analysis that will be used throughout the paper.

Let $V$ be a reflexive Banach space equipped with the norm $\| \cdot \|_V$.
The topological dual space of $V$ is denoted by $V^*$, and $\left< \cdot , \cdot \right>_{V^* \times V}$ denotes the duality pairing  of $V$, i.e.,
\begin{equation*}
\left< p, u \right>_{V^* \times V} = p(u), \quad u \in V, \gap p \in V^*.
\end{equation*}
We may omit the subscripts if there is no ambiguity.
We denote the collection of proper, convex, lower semicontinuous functionals from $V$ to $\overline{\mathbb{R}}$ by $\Gamma_0 (V)$.

The \textit{effective domain} of a proper functional $F$:~$V \rightarrow \overline{\mathbb{R}}$ is denoted by $\dom F$, i.e.,
\begin{equation*}
\dom F = \left\{ u \in V : F(u) < \infty \right\}.
\end{equation*}
For example, for a subset $K$ of $V$, its \textit{characteristic function} $\chi_K$:~$V \rightarrow \overline{\mathbb{R}}$ defined by
\begin{equation}
\label{chi}
\chi_K (u) = \begin{cases} 0 & \textrm{ if } u \in K, \\ \infty & \textrm{ if } u \not\in K \end{cases}
\end{equation}
has the effective domain $\dom \chi_K = K$.

A functional $F$:~$V \rightarrow \overline{\mathbb{R}}$ is said to be \textit{coercive} if
\begin{equation*}
F(u) \rightarrow \infty \quad\textrm{as}\quad \| u \| \rightarrow \infty.
\end{equation*}
If $F \in \Gamma_0 (V)$ is coercive, then the minimization problem
\begin{equation}
\label{min_F}
\min_{u \in V} F(u)
\end{equation}
has a solution $u^* \in V$ with $F(u^*) > -\infty$~\cite[Proposition~11.14]{BC:2011}.
If we further assume that $F$ is strictly convex, then the solution of~\cref{min_F} is unique.

For a convex functional $F$:~$V \rightarrow \overline{\mathbb{R}}$, the \textit{subdifferential} of $F$ at a point $u \in V$ is defined as
\begin{equation*}
\partial F(u) = \left\{ p \in V^* : F(v) \geq F(u) + \left< p , v- u \right> \quad \forall v \in V \right\}.
\end{equation*}
If $F$ is Frech\'{e}t differentiable at $u$, then the subdifferential $\partial F (u)$ agrees with the Frech\'{e}t derivative $F'(u)$, i.e., $\partial F(u) = \{ F'(u) \}$.
It is clear from the definition of subdifferential that $u^* \in V$ is a global minimizer of $F$ if and only if $0 \in \partial F(u^*)$.

If $F_k$:~$V \rightarrow \overline{\mathbb{R}}$, $1 \leq k \leq N$ are proper convex functionals, one can obtain directly from the definition of subdifferential that
\begin{equation}
\label{subdifferential1}
\partial \left( \sumk F_k \right) (u) \supseteq \sumk \partial F_k (u), \quad u \in V.
\end{equation} 
We have a similar result on the composition with a linear operator; let $W$ be a reflexive Banach space.
For a proper convex functional $F$:~$V \rightarrow \overline{\mathbb{R}}$ and a bounded linear functional $A$:~$W \rightarrow V$, one can show that
\begin{equation}
\label{subdifferential2}
\partial (F \circ A)(w) \supseteq A^* \partial F (Aw), \quad w \in W.
\end{equation}

The \textit{Legendre--Fenchel conjugate} $F^*$:~$V^* \rightarrow \overline{\mathbb{R}}$ of a functional $F$:~$V \rightarrow \overline{\mathbb{R}}$ is defined by
\begin{equation*}
F^*(p) = \sup_{u \in V} \left\{ \left< p, u \right> - F(u) \right\}.
\end{equation*}
Clearly, $F^*$ is convex lower semicontinuous regardless of whether $F$ is.
If we further assume that $F \in \Gamma_0 (V)$, then $\partial F$ and $\partial F^*$ are inverses of each other~\cite[Theorem~16.23]{BC:2011}, i.e., we have
\begin{equation}
\label{Legendre}
p \in \partial F(u) \gap \Leftrightarrow \gap u \in \partial F^* (p), \quad u \in V, \gap p \in V^*.
\end{equation}

For convex functionals $F_k$, $1 \leq k \leq N$, defined on $V$, the \textit{infimal convolution} of $F_k$ is given by
\begin{equation*}
\left( \bigsquare_{k=1}^N F_k \right) (v) = \inf \left\{ \sumk F_k (v_k ) : v = \sumk v_k, \gap v_k \in V \right\}.
\end{equation*}
It is easy to check that $\bigsquareinline_{k=1}^N F_k$ is convex.
If each $F_k$ is in $\Gamma_0 (V)$ and coercive, then we have $\bigsquareinline_{k=1}^N F_k \in \Gamma_0 (V)$~\cite[Proposition~12.14]{BC:2011}.

For another reflexive Banach space $W$ and a bounded linear operator $A$:~$V \rightarrow W$, the \textit{infimal postcomposition} $A \triangleright F$:~$W \rightarrow \overline{\mathbb{R}}$ of a convex functional $F$:~$V \rightarrow \overline{\mathbb{R}}$ by $A$ is given by
\begin{equation*}
\left( A \triangleright F \right)(w) = \inf \left\{ F(v) : Av = w, \gap v \in V \right\}.
\end{equation*}
If there does not exist $v \in V$ such that $Av = w$, then we set $(A \triangleright F)(w) = \infty$.
One can show that $A \triangleright F$ is also convex~\cite[Proposition~12.34]{BC:2011}.
If the adjoint $A^*$:~$W^* \rightarrow V^*$ of $A$ is surjective and $F \in \Gamma_0 (V)$, then we get $A \triangleright F \in \Gamma_0 (W)$~\cite[Lemma~2.6]{BC:2013}.
We have the following formulas for the convex conjugates for infimal convolution and infimal postcomposition~\cite[Proposition~13.21]{BC:2011}:
\begin{equation}
\label{dual_inf}
\left( \bigsquare_{k=1}^N F_k \right)^* = \sumk F_k^* \quad \textrm{and} \quad
(A \triangleright F)^* = F^* \circ A^*.
\end{equation}
We state a useful identity on infimal convolution and infimal postcomposition in \cref{Lem:infimal}, whose proof will be given in \cref{App:proof_infimal}.

\begin{lemma}
\label{Lem:infimal}
For a positive integer $N$, let $W_k$, $1 \leq k \leq N$, and $W$ be real vector spaces.
For linear operators $A_k$:~$W_k \rightarrow W$ and functionals $F_k$:~$W_k \rightarrow \overline{\mathbb{R}}$,  the following is satisfied:
\begin{equation*}
\left( \bigsquare_{k=1}^N (A_k \triangleright F_k) \right) (w) = \inf \left\{ \sumk F_k (w_k) : w = \sumk A_k w_k, \gap w_k \in W_k \right\}, \quad w \in W.
\end{equation*}
\end{lemma}

For a convex and Frech\'{e}t differentiable functional $F$:~$V \rightarrow \mathbb{R}$, the \textit{Bregman distance} of $F$ is defined by
\begin{equation*}
D_F (u , v) = F(u) - F(v) - \left< F'(v), u-v \right>, \quad u,v \in V.
\end{equation*}
Note that $D_F$ is convex and Frech\'{e}t differentiable with respect to its first argument, i.e., for fixed $v \in V$ the map $u \mapsto D_F (u, v)$ is Frech\'{e}t differentiable and convex.

\begin{remark}
\label{Rem:Hilbert}
Although the results in the references~\cite{BC:2011,BC:2013} we cited in this section are stated in the Hilbert space setting, they are still valid for reflexive Banach spaces.
Two main properties of Hilbert spaces used in~\cite{BC:2011,BC:2013} are the weak compactness of closed bounded sets and the equivalence between the strong and weak lower semicontinuity of convex functions, and they are also true for reflexive Banach spaces.
\end{remark}

\section{Gradient methods}
\label{Sec:Gradient} 
In this section, we propose an abstract gradient method that generalizes several existing first order methods for convex optimization.
As we will see in~\cref{Sec:ASM}, conventional additive Schwarz methods for convex optimization are interpreted as abstract gradient methods.
Therefore, the abstract gradient method and its convergence proof shall be very useful in the analysis of additive Schwarz methods.

Throughout this section, let $V$ be a reflexive Banach space.
We consider the following model problem:
\begin{equation}
\label{model_gradient}
\minu \left\{ E(u):= F(u) + G(u) \right\},
\end{equation}
where $F$:~$V\rightarrow \mathbb{R}$ is a Frech\'{e}t differentiable convex function and $G \in \Gamma_0 (V)$ is possibly nonsmooth.
We further assume that $E$ is coercive, so that a solution $u^* \in V$ of~\cref{model_gradient} exists.
The optimality condition of $u^*$ reads as
\begin{equation*}
F'(u^*) + \partial G(u^*) \ni 0,
\end{equation*}
or equivalently,
\begin{equation}
\label{optimality}
\left< F'(u^*) , u - u^* \right> + G(u) - G(u^* ) \geq 0, \quad u \in V.
\end{equation}

Let $B$:~$V \times V \rightarrow \overline{\mathbb{R}}$ be a functional which is proper, convex, and lower semicontinuous with respect to its first argument.
In addition, we assume that $B$ satisfies the following.

\begin{assumption}
\label{Ass:gradient}
There exists constants $q > 1$ and $\theta \in (0, 1]$ such that for any bounded and convex subset $K$ of $V$, we have 
\begin{multline*}
D_F (u, v) + G(u) \leq B (u, v) \\
\leq \frac{L_K}{q} \| u - v \|^q + \theta G \left( \frac{1}{\theta}u - \left( \frac{1}{\theta} - 1 \right) v \right) + (1 - \theta ) G(v) , \quad u,v \in K \cap \dom G,
\end{multline*}
where $L_K$ is a positive constant depending on $K$.
\end{assumption}

With the functional $B$ satisfying \cref{Ass:gradient}, the abstract gradient method for~\cref{model_gradient} is presented in \cref{Alg:gradient}.

\begin{algorithm}[]
\caption{Abstract gradient method for~\cref{model_gradient}}
\begin{algorithmic}[]
\label{Alg:gradient}
\STATE Choose $u^{(0)} \in \dom G$.
\FOR{$n=0,1,2,\dots$}
\item \vspace{-0.5cm}\begin{equation*}
\unn \in \argmin_{u \in V} \left\{ Q(u, \un):= F(\un) + \langle F'(\un), u - \un \rangle + B(u, \un ) \right\}
\end{equation*}\vspace{-0.4cm}
\ENDFOR
\end{algorithmic}
\end{algorithm}

Several fundamental first order methods for~\cref{model_gradient} can be represented as examples of \cref{Alg:gradient}.
Under the assumption that $F'$ is Lipschitz continuous with modulus $M > 0$, setting
\begin{equation*}
B(u,v) = \frac{1}{2\tau}\|u - v\|^2 + G(u)
\end{equation*}
for $\tau \in (0, 1/M]$ satisfies \cref{Ass:gradient} with $q = 2$, $\theta = 1$, $L_{K} = 1/\tau$ and yields the forward-backward splitting algorithm~\cite{BT:2009} for~\cref{model_gradient}.
If we further assume that $G = 0$, then it reduces to the classical fixed-step gradient descent method.

First, we claim that the energy of \cref{Alg:gradient} always decreases under \cref{Ass:gradient} in the following lemma; the proof will be given in \cref{App:proof_decreasing}.

\begin{lemma}
\label{Lem:decreasing}
Suppose that \cref{Ass:gradient} holds.
In \cref{Alg:gradient}, the sequence $\{ E(\un )\}$ is decreasing.
\end{lemma}

We note that $E (u^{(0)}) < \infty$ because $u^{(0)} \in \dom G$.
By \cref{Lem:decreasing} and the coercivity of $E$, the sequence $\{ \un \}$ generated by \cref{Alg:gradient} is contained in the bounded set
\begin{equation}
\label{K0}
K_0 = \left\{ u \in V : E(u) \leq E(u^{(0)}) \right\}.
\end{equation}
Clearly, $K_0$ is convex and $K_0 \subseteq \dom G$.
Since $K_0$ is bounded, there exists a constant $R_0 > 0$ such that
\begin{equation}
\label{R0}
K_0 \subseteq \left\{ u \in V : \| u - u^* \| \leq R_0 \right\}.
\end{equation}
In what follows, we omit the subscript $K_0$ from $L_{K_0}$ and write $L = L_{K_0}$.

We describe the convergence behavior of \cref{Alg:gradient}.
Although \cref{Alg:gradient} is written in a fairly general fashion, its convergence analysis can be done in a similar way to the vanilla gradient method described in~\cite{Nesterov:2013}.
The proof of the following convergence theorem for \cref{Alg:gradient} can be found in \cref{App:proof_gradient}.

\begin{theorem}
\label{Thm:gradient}
Suppose that \cref{Ass:gradient} holds.
In \cref{Alg:gradient}, if $E(u^{(0)}) - E(u^*) \geq \theta^{q-1} LR_0^q$, then
\begin{equation*}
E(u^{(1)}) - E(u^*) \leq \left( 1 - \theta \left( 1- \frac{1}{q} \right)\right) ( E(u^{(0)}) - E(u^*) ).
\end{equation*}
Otherwise, we have
\begin{equation*}
E(u^{(n)}) - E(u^*) \leq \frac{C_{q,\theta} L R_0^q}{(n+1)^{q-1}}, \quad n \geq 0,
\end{equation*}
where $C_{q,\theta}$ is a positive constant defined in~\cref{Cq} depending on $q$ and $\theta$ only, and $R_0$ was defined in~\cref{R0}.
\end{theorem}

\Cref{Thm:gradient} means that the convergence rate of the energy error of \cref{Alg:gradient} is $O(1/n^{q-1})$. 
If the functional $E$ in~\cref{model_gradient} is \textit{sharp}, then a better convergence rate can be obtained.
The sharpness condition of $F$ is summarized in \cref{Ass:sharp}.

\begin{assumption}[sharpness]
\label{Ass:sharp}
There exists a constant $p > 1$ such that for any bounded and convex subset $K$ of $V$ satisfying $u^* \in K$, we have
\begin{equation*}
\frac{\mu_K}{p} \| u - u^* \|^{p} \leq E(u) - E(u^*), \quad u \in K,
\end{equation*}
for some $\mu_K > 0$.
\end{assumption}

The inequality in \cref{Ass:sharp} is also known as the {\L}ojasiewicz inequality.
It is known that quite many kinds of functions satisfy \cref{Ass:sharp}; see~\cite{BDL:2007,XY:2013}.
Invoking~\cref{optimality}, one can obtain the following simple criterion to check whether \cref{Ass:sharp} holds.

\begin{proposition}
\label{Prop:uniform}
Consider the minimization problem~\cref{model_gradient}.
For any bounded and convex subset $K$ of $V$, assume that $F$ is uniformly convex with parameters $p$ and $\mu_K$ on $K$, i.e.,
\begin{equation}
\label{uniform}
D_F (u,v) \geq \frac{\mu_K}{p} \| u - v\|^p, \quad u,v \in K. 
\end{equation}
Then \cref{Ass:sharp} holds.
\end{proposition}

We write $\mu = \mu_{K_0}$, where $K_0$ was defined in~\cref{K0}.
One can prove without major difficulty that $p$ should be greater than or equal to $q$ in order to satisfy  \cref{Ass:gradient,Ass:sharp} simultaneously.
Note that under \cref{Ass:sharp}, a solution of~\cref{model_gradient} is unique.

With \cref{Ass:gradient,Ass:sharp}, the following improved convergence theorem for \cref{Alg:gradient} is available; see \cref{App:proof_gradient_uniform} for the proof.

\begin{theorem}
\label{Thm:gradient_uniform}
Suppose that \cref{Ass:gradient,Ass:sharp} hold.
In \cref{Alg:gradient}, we have the following:
\begin{enumerate}
\item In the case $p = q$, we have
\begin{equation*}
\resizebox{0.9\textwidth}{!}{$ \displaystyle E(\un) - E(u^*) \leq \left( 1 - \left( 1- \frac{1}{q}  \right) \min \left\{ \theta , \left( \frac{\mu}{qL} \right)^{\frac{1}{q-1}} \right\} \right)^n ( E(u^{(0)}) - E(u^*) ), \gap n \geq 0.$}
\end{equation*}
\item In the case $p > q$, if $E(u^{(0)}) - E(u^*) \geq \theta^{\frac{p(q-1)}{p-q}} p^{\frac{q}{p-q}} (L^p / \mu^q )^{\frac{1}{p-q}}$, then
\begin{equation*}
E(u^{(1)}) - E(u^* ) \leq \left( 1 - \theta \left( 1 - \frac{1}{q} \right) \right) ( E(u^{(0)}) - E(u^*) ).
\end{equation*}
Otherwise, we have
\begin{equation*}
E(u^{(n)}) - E(u^*) \leq \frac{C_{p,q,\theta}(L^p/\mu^q)^{\frac{1}{p-q}}}{(n+1)^{\frac{p(q-1)}{p-q}}}, \quad n \geq 0,
\end{equation*}
where $C_{p,q,\theta}$ is a positive constant defined in~\cref{Cpq} depending on $p$, $q$, and $\theta$ only.
\end{enumerate}
\end{theorem}

In \Cref{Thm:gradient,Thm:gradient_uniform}, the decay rate of the energy error $E(\un ) - E(u^*)$ depends on only $p$, $q$, $\theta$, $L$, and $\mu$ if the initial energy error $E (u^{(0)}) - E(u^*)$ is small enough.
Therefore, in applications, it is enough to estimate those variables to get the convergence rate of the algorithm.

\section{Additive Schwarz methods for convex optimization}
\label{Sec:ASM}
This section is devoted to an abstract convergence theory of additive Schwarz methods for convex optimization~\cref{model_gradient}.
We present an additive Schwarz method for~\cref{model_gradient} based on an abstract framework of space decomposition.
Then we show that the proposed method is an instance of~\cref{Alg:gradient}.
Such an observation makes the convergence analysis of the proposed method straightforward.

First, we present a space decomposition setting.
Throughout this section, an index $k$ runs from $1$ to $N$.
Let $V_k$ be a reflexive Banach space and $R_k^*$:~$V_k \rightarrow V$ be a bounded linear operator such that
\begin{equation}
\label{decomposition}
V = \sumk R_k^* V_k
\end{equation}
and its adjoint $R_k$:~$V^* \rightarrow V_k^*$ is surjective.
For example, if $V_k$ is a subspace of $V$, then one may choose $R_k^*$ as the natural embedding.

In our framework, we allow inexact local solvers.
Let $d_k$:~$V_k \times V \rightarrow \overline{\mathbb{R}}$ and $G_k$:~$V_k \times V \rightarrow \overline{\mathbb{R}}$ be functionals which are proper, convex, and lower semicontinuous with respect to their first arguments.
Local problems of the proposed method shall have the following general form:
\begin{equation}
\label{local_general}
\min_{w_k \in V_k} \left\{ F(v) + \left< F'(v), R_k^* w_k \right> + \omega d_k (w_k, v) + G_k (w_k , v) \right\} 
\end{equation}
for some $v \in V$ and $\omega > 0$.
In the case of exact local solvers, we set
\begin{equation} \label{exact_local} \begin{split}
d_k (w_k, v) &= D_F (v + \Rw , v), \\
G_k (w_k, v) &= G(v + \Rw )
\end{split} \end{equation}
for $w_k \in V_k$, $v \in V$, and we set $\omega = 1$.
Then~\cref{local_general} becomes
\begin{equation*}
\min_{w_k \in V_k} E(v + \Rw ).
\end{equation*}

We present a general additive Schwarz method for~\cref{model_gradient} with local problems~\cref{local_general} in \cref{Alg:ASM}.
The constants $\tau_0$ and $\omega_0$ in \cref{Alg:ASM} will be defined in \cref{Ass:convex,Ass:local}, respectively.

\begin{algorithm}[]
\caption{Additive Schwarz method for~\cref{model_gradient}}
\begin{algorithmic}[]
\label{Alg:ASM}
\STATE Choose $u^{(0)} \in \dom G$, $\tau \in (0, \tau_0 ]$, and $\omega \geq \omega_0$.
\FOR{$n=0,1,2,\dots$}
\item \vspace{-0.5cm} \begin{equation*}
\resizebox{0.9\textwidth}{!}{$ \displaystyle \begin{split}
w_k^{(n+1)} &\in \argmin_{w_k \in V_k} \left\{ F(\un) + \langle F'(\un), \Rw \rangle
 + \omega d_k (w_k, \un ) + G_k ( w_k, \un) \right\}, \gap 1 \leq k \leq N, \\
\unn &= \un + \tau \sumk \Rw^{(n+1)}
\end{split} $}
\end{equation*} \vspace{-0.4cm}
\ENDFOR
\end{algorithmic}
\end{algorithm}

In order to ensure convergence of \cref{Alg:ASM}, the following three conditions should be considered: stable decomposition, strengthened convexity, and local stability.

\begin{assumption}[stable decomposition]
\label{Ass:stable}
There exists a constant $q > 1$ such that for any bounded and convex subset $K$ of $V$, the following holds:
for any $u, v \in K \cap \dom G$, there exists $w_k \in V_k$, $1\leq k \leq N$, such that
\begin{equation*}
u-v = \sumk \Rw ,
\end{equation*}
\begin{equation*}
\sumk d_k (w_k, v) \leq \frac{C_{0,K}^q}{q} \| u-v \|^q ,
\end{equation*}
and
\begin{equation}
\label{stable_nonsmooth}
\sumk G_k ( w_k, v) \leq G \left( u \right) + (N-1) G(v),
\end{equation}
where $C_{0, K}$ is a positive constant depending on $K$.
\end{assumption}

Similar assumptions to \cref{Ass:stable} for Schwarz methods can be found in existing works, e.g.,~\cite[Assumption~1 and equation~(7)]{BK:2012}.
In those works, several function decompositions tailored for particular applications were proposed.
We will see in \cref{Sec:Applications} that \cref{Ass:stable} is compatible with them.
We also note that the assumption~\cref{stable_nonsmooth} for the nonsmooth part $G$ of~\cref{model_gradient} is essential; a counterexample for the convergence of Schwarz methods for a problem not satisfying~\cref{stable_nonsmooth} was introduced in~\cite[Claim~6.1]{LN:2017}. 

\begin{assumption}[strengthened convexity]
\label{Ass:convex}
There exists a constant $\tau_0 \in (0, 1]$ which satisfies the following:
for any $v \in V$, $w_k \in V_k$, $1 \leq k \leq N$, and $\tau \in (0, \tau_0 ]$, we have
\begin{equation*}
\left( 1 - \tau N \right) E(v) + \tau \sumk E (v + \Rw) \geq E \left( v + \tau \sumk \Rw \right).
\end{equation*}
\end{assumption}

By the convexity of $E$, \cref{Ass:convex} is valid $\tau_0 = 1/N$.
However, a smaller value for $\tau_0$ independent of $N$ can be found by, for example, the coloring technique; details will be given in \cref{Sec:DD}.

\begin{assumption}[local stability]
\label{Ass:local}
There exists a constant $\omega_0 > 0$ which satisfies the following:
for any $v \in \dom G$, and $w_k \in V_k$, $1 \leq k \leq N$, we have
\begin{align*}
D_F ( v + \Rw, v ) &\leq \omega_0 d_k (w_k, v), \\
G(v + \Rw ) &\leq G_k (w_k , v).
\end{align*}
\end{assumption}

In the case of exact solvers, i.e.,~\cref{exact_local}, \cref{Ass:local} is trivial with $\omega_0 = 1$.
In general, as explained in~\cite{TW:2005}, \cref{Ass:local} gives a one-sided measure of approximation properties of the local solvers.
One can use any local solvers satisfying \cref{Ass:local} for \cref{Alg:ASM}.

For convergence analysis of \cref{Alg:ASM}, we introduce a functional $\M : V \times V \rightarrow \overline{\mathbb{R}}$:
for two positive real numbers $\tau$ and $\omega$, the functional $\M$ is defined as
\begin{equation}
\label{M}
\resizebox{\textwidth}{!}{$\displaystyle \M (u, v) = \tau \inf \left\{ \sumk (\omega d_k + G_k) (w_k, v)  : u-v = \tau \sumk \Rw, \gap w_k \in V_k \right\}
 + \left(1- \tau N \right) G( v), \gap u, v \in V.$}
\end{equation}
The following lemma summarizes important properties of $\M$.

\begin{lemma}
\label{Lem:M}
For $\tau$, $\omega > 0$, the functional $\M$:~$V \times V \rightarrow \mathbb{R}$ defined in~\cref{M} is convex and lower semicontinuous with respect to its first argument.
\end{lemma}
\begin{proof}
For convenience, we fix $v \in V$ and write
\begin{equation*}
\M (u)  = \M(u,v), \quad d_k (w_k) = d_k (w_k, v), \quad G_k (w_k) = G_k (w_k, v) 
\end{equation*}
for $u \in V$ and $w_k \in V_k$.
By \cref{Lem:infimal} we have
\begin{equation} 
\label{M1}
\resizebox{\textwidth}{!}{$ \displaystyle \begin{split} \M (u) &= \tau \inf \left\{ \sumk (\omega d_k + G_k ) (w_k) : u- v = \tau \sumk \Rw, \gap w_k \in V_k\right\} + (1- \tau  N) G(v) \\
&= \tau \left( \bigsquare_{k=1}^N \left((\tau R_k^*) \triangleright (\omega d_k + G_k ) \right) \right) (u-v ) + (1- \tau  N) G(v). \end{split} $}
 \end{equation}
Since $R_k$ is surjective, by~\cite[Lemma~2.6]{BC:2013} we get the desired result.
\end{proof}

The following lemma, named the \textit{generalized additive Schwarz lemma}, shows that \cref{Alg:ASM} in fact belongs to a class of \cref{Alg:gradient} with $B(u,v) = \M (u,v)$.

\begin{lemma}[generalized additive Schwarz lemma]
\label{Lem:ASM}
Let $\{ \un \}$ be the sequence generated by \cref{Alg:ASM}.
Then it satisfies
\begin{equation}
\label{unn_ASM}
\unn \in \argmin_{u \in V} \left\{ F(\un) + \langle F'(\un), u - \un \rangle + \M (u, u^{(n)}) \right\}, \quad n \geq 0,
\end{equation}
where $\M (u, \un)$ was given in~\cref{M}.
\end{lemma}
\begin{proof}
Choose any $n \geq 0$.
We write
\begin{align*}
\dkn (w_k) &= (\omega d_k + G_k ) (w_k, \un ), \quad w_k \in V_k, \\
\Mn (u) &= \M (u, \un), \quad\quad\quad\quad\quad u \in V
\end{align*}
for convenience.
The optimality condition of~\cref{unn_ASM} is given by
\begin{equation*}
F'(\un) + \partial \Mn (u) \ni 0,
\end{equation*}
or equivalently,
\begin{equation*}
u \in \partial M_{\tau, \omega}^{(n)*}(- F'(\un))
\end{equation*}
by~\cref{Legendre} and \cref{Lem:M}.
Therefore, it suffices to show that
\begin{equation}
\label{ASM_WTS}
\unn \in \partial M_{\tau, \omega}^{(n)*}(- F'(\un)).
\end{equation}

The optimality condition for $w_k^{(n+1)}$ reads as
\begin{equation*}
R_k F'(\un) + \partial \dkn(w_k^{(n+1)}) \ni 0.
\end{equation*}
Since $\dkn \in \Gamma_0 (V_k)$, one can obtain the following explicit formula for $w_k^{(n+1)}$:
\begin{equation}
\label{wk}
w_k^{(n+1)} \in \partial \dkns ( - R_k F'(\un) ).
\end{equation}
Summation of~\cref{wk} over $1\leq k \leq N$ yields
\begin{equation}
\label{ASM1}
\unn \in \un + \tau \sumk R_k^*  \partial \dkns  ( - R_k F'(\un) ).
\end{equation}
On the other hand, dualizing~\cref{M1} with $v = \un$ yields
\begin{equation}
\label{ASM_preconditioner} \begin{split}
M_{\tau,\omega}^{(n)*} ( p) &= \tau \left( \bigsquare_{k=1}^N ( (\tau R_k^*) \triangleright \dkn ) \right)^* \left(\frac{1}{\tau }  p\right) + \langle p, \un \rangle \\
&\stackrel{\cref{dual_inf}}{=} \tau \left( \sumk \dkns \circ (\tau R_k) \right) \left( \frac{1}{\tau} p \right) + \langle p, \un \rangle \\
&= \tau \sumk \dkns \left( R_k p \right) + \langle p, \un \rangle
\end{split} \end{equation}
for $p \in V^*$.
Consequently, by~\cref{subdifferential1,subdifferential2} we have
\begin{equation}
\label{ASM2}
\partial M_{\tau,\omega}^{(n)*}  (p) \supseteq \tau \sumk R_k^* \partial \dkns \left( R_k p \right) + \un.
\end{equation}
If we substitute $p$ by $- F'(\un)$ in~\cref{ASM2}, we get
\begin{equation}
\label{ASM3}
\partial M_{\tau,\omega}^{(n)*} (-F' (\un) ) \supseteq \tau \sumk R_k^*  \partial \dkns ( -R_k F' (\un) ) + \un.
\end{equation}
Combining~\cref{ASM1,ASM3}, we have~\cref{ASM_WTS}.
\end{proof}

Thanks to \cref{Lem:ASM}, it suffices to verify that \cref{Ass:gradient} holds when $B(u,v) = \M (u,v)$ in order to show the convergence of \cref{Alg:ASM}.
In the following, we prove that \cref{Ass:stable,Ass:convex,Ass:local} are sufficient to ensure \cref{Ass:gradient}.

\begin{lemma}
\label{Lem:M_gradient}
Suppose that \cref{Ass:stable,Ass:convex,Ass:local} hold.
Let $\tau \in (0, \tau_0]$ and $\omega \geq \omega_0$.
For any bounded and convex subset $K$ of $V$, we have
\begin{multline*}
D_F (u, v) + G(u) \leq \M (u, v) \\
\leq \frac{\omega C_{0,K'}^q}{q \tau^{q-1}} \| u -v \|^q + \tau G \left( \frac{1}{\tau}u - \left( \frac{1}{\tau} - 1 \right) v \right) + (1-\tau ) G(v), \quad u,v \in K \cap \dom G,
\end{multline*}
where the functional $M_{\tau, \omega}$ was given in~\cref{M} and
\begin{equation}
\label{K'}
K' = \left\{ \frac{1}{\tau} u - \left( \frac{1}{\tau} - 1 \right) v : u,v \in K \right\}.
\end{equation}
\end{lemma}
\begin{proof}
Take any $w_k \in V_k$ such that
\begin{equation}
\label{M_gradient1}
u - v = \tau \sumk \Rw.
\end{equation}
By \cref{Ass:local} we get
\begin{subequations} \label{M_gradient2}
\begin{equation} 
\begin{split}
\tau \sumk \omega d_k (w_k , v) &\geq \tau \sumk D_F (v + \Rw, v) \\
&= \tau\sumk F(v + \Rw) - \tau N F(v) - \left< F'(v), u-v \right>
\end{split} \end{equation}
and
\begin{equation}
\tau \sumk G_k (w_k, v) \geq \tau \sumk G( v + \Rw ).
\end{equation}
\end{subequations}
Then using \cref{Ass:convex}, we have
\begin{equation*} \begin{split}
&\tau \sumk (\omega d_k + G_k) (w_k, v)  + \left( 1 -  \tau N \right) G(v) \\
&\quad\quad \stackrel{\cref{M_gradient2}}{\geq} \left( 1 - \tau N \right) E(v) + \tau \sumk E (v + \Rw) - F(v) - \left< F'(v) , u-v \right> \\
&\quad\quad\geq E(u) - F(v) - \left< F'(v), u-v \right> \\
&\quad\quad = D_F (u, v) + G(v) .
\end{split} \end{equation*}
Taking the infimum on the left-hand side of the above equation over all $w_k$ satisfying~\cref{M_gradient1} yields
\begin{equation*}
D_F (u, v) + G(u) \leq \M (u, v).
\end{equation*}

On the other hand, let
\begin{equation*}
\bar{u} = \frac{1}{\tau}u - \left( \frac{1}{\tau} - 1 \right) v.
\end{equation*}
Since $\bar{u}, v \in K'$, by \cref{Ass:stable}, there exist $\bar{w}_k \in V_k$, $1 \leq k \leq N$, such that
\begin{equation}
\label{M_gradient3}
\bar{u} - v = \sumk R_k^* \bar{w}_k,
\end{equation}
\begin{equation}
\label{M_gradient4}
\sumk d_k (\bar{w}_k , v) \leq \frac{C_{0,K'}^q}{q} \left\|  \bar{u} - v \right\|^q ,
\end{equation}
and
\begin{equation*}
\sumk G_k (\bar{w}_k, v) \leq G(\bar{u} ) + (N-1) G(v).
\end{equation*}
Note that
\begin{equation*}
u-v = \tau (\bar{u} - v ) = \tau \sumk R_k^* \bar{w}_k.
\end{equation*}
By~\cref{M_gradient3,M_gradient4}, it follows that
\begin{equation*} \begin{split}
\M (u,v) &\leq \tau \sumk ( \omega d_k + G_k ) (\bar{w}_k, v ) + ( 1- \tau N ) G(v) \\
&\leq \frac{\tau \omega C_{0,K'}^q }{q} \left\| \bar{u} - v \right\|^q + \tau G(\bar{u}) + (1-\tau ) G(v) \\
&= \frac{\omega C_{0,K'}^q}{q \tau^{q-1}} \| u -v \|^q + \tau G \left( \frac{1}{\tau}u - \left( \frac{1}{\tau} - 1 \right) v \right) + (1-\tau ) G(v).
\end{split} \end{equation*}
Now, the proof is complete.
\end{proof}

\Cref{Lem:M_gradient} means that  \cref{Alg:ASM} satisfies \cref{Ass:gradient} under \cref{Ass:stable,Ass:convex,Ass:local} with $\theta = \tau$, $L_K = \omega C_{0,K'}^q / \tau^{q-1}$.
Then by \cref{Lem:decreasing}, the energy sequence $\{ E( \un )\}$ generated by \cref{Alg:ASM} always decreases.
Thus, if we define the set $K_0 \subseteq V$ as~\cref{K0}, the sequence $\{ \un \}$ is contained in $K_0$.
Recall that we can choose $R_0 > 0$ satisfying~\cref{R0}.
In the following, we write $C_0 = C_{0,K_0'}$, where $K_0'$ is defined in the same way as~\cref{K'}.
If $F$ additionally satisfies \cref{Ass:sharp}, we write $\mu = \mu_{K_0}$.
We define the \textit{additive Schwarz condition number} $\kASM$ as follows:
\begin{equation}
\label{kASM}
\kASM = \frac{\omega C_0^q}{\tau^{q-1}}.
\end{equation}
Then the value of $\kASM$ depends on $\tau$, $\omega$, and $u^{(0)}$.
By \cref{Thm:gradient,Thm:gradient_uniform}, the following convergence theorems for \cref{Alg:ASM} are straightforward.

\begin{theorem}
\label{Thm:ASM}
Suppose that \cref{Ass:stable,Ass:convex,Ass:local} hold.
In \cref{Alg:ASM}, if $E(u^{(0)}) - E(u^*) \geq  \tau^{q-1} R_0^q \kASM$, then
\begin{equation*}
E(u^{(1)}) - E(u^*) \leq \left( 1 - \tau \left( 1 - \frac{1}{q} \right) \right) ( E(u^{(0)}) - E(u^*) ).
\end{equation*}
Otherwise, we have
\begin{equation*}
E(u^{(n)}) - E(u^*) \leq \frac{C_{q, \tau} R_0^q \kASM}{(n+1)^{q-1}}, \quad n \geq 0,
\end{equation*}
where $C_{q, \tau}$ is a positive constant defined in~\cref{Cq} depending on $q$ and $\tau$ only, $R_0$ was defined in~\cref{R0}, and $\kASM$ was defined in~\cref{kASM}.
\end{theorem}

\begin{theorem}
\label{Thm:ASM_uniform}
Suppose that \cref{Ass:stable,Ass:convex,Ass:local,Ass:sharp} hold.
In \cref{Alg:ASM}, we have the following:
\begin{enumerate}
\item In the case $p = q$, we have
\begin{equation*}
\resizebox{0.9\textwidth}{!}{$ \displaystyle E(\un) - E(u^*) \leq \left(1 - \left( 1- \frac{1}{q} \right) \min \left\{ \tau, \left( \frac{\mu}{q \kASM} \right)^{\frac{1}{q-1}} \right\} \right)^n ( E(u^{(0)}) - E(u^*) ), \gap n \geq 0.$}
\end{equation*}
\item In the case $p > q$, if $E(u^{(0)}) - E(u^*) \geq p^{\frac{q}{p-q}} \tau^{\frac{p(q-1)}{p-q}} (\kASM^p / \mu^q)^{\frac{1}{p-q}}$, then
\begin{equation*}
E(u^{(1)}) - E(u^* ) \leq \left(1 - \tau \left( 1- \frac{1}{q} \right)\right) ( E(u^{(0)}) - E(u^*) ).
\end{equation*}
Otherwise, we have
\begin{equation*}
E(u^{(n)}) - E(u^*) \leq \frac{C_{p,q,\tau} \left( \kASM^p / \mu^q \right)^{\frac{1}{p-q}}}{(n+1)^{\frac{p(q-1)}{p-q}}}, \quad n \geq 0,
\end{equation*}
where $C_{p,q,\tau}$ is a positive constant defined in~\cref{Cpq} depending on $p$, $q$, and $\tau$ only, and $\kASM$ was defined in~\cref{kASM}.
\end{enumerate}
\end{theorem}

In \cref{Thm:ASM,Thm:ASM_uniform}, we observe that the asymptotic convergence rate of \cref{Alg:ASM} becomes faster as $\kASM$ becomes smaller.
Therefore, getting sharp estimates for $C_0$, $\tau_0$, and $\omega_0$ is important in the analysis of additive Schwarz methods.
We will consider in \cref{Sec:Applications} how to estimate those constants.

\subsection{Relation to the classical additive Schwarz theory}
Here, we show that the additive Schwarz framework proposed in this section is a generalization of the classical theory for linear elliptic problems developed in~\cite{TW:2005}.
Throughout this section, $H$ denotes a Hilbert space.

Let $a(\cdot , \cdot )$:~$H \times H \rightarrow \mathbb{R}$ be a continuous and symmetric positive definite~(SPD) bilinear form on $H$, and $f \in H^*$.
We consider the variational problem
\begin{equation*}
a(u, v) = \left< f, v \right>, \quad v \in H.
\end{equation*}
The above problem is standard in the field of elliptic partial differential equations.
By the Lax--Milgram theorem, a unique solution $u^* \in H$ of the above problem is characterized as a solution of the minimization problem
\begin{equation}
\label{model_linear}
\min_{u \in  H} \left\{ F(u) := \frac{1}{2} a (u, u) - \left<f , u \right> \right\}.
\end{equation}
If one has
\begin{equation*}
a(u,v) = \left< Au, v \right>, \quad u, v \in H
\end{equation*}
for some continuous and SPD linear operator $A$:~$H \rightarrow H^*$, the energy functional $F(u)$ in~\cref{model_linear} is Frech\'{e}t differentiable with the derivative $F'(u) = Au - f \in H^*$ for $u \in H$.
Hence,~\cref{model_linear} is a particular instance of~\cref{model_gradient} and the theory developed in \cref{Sec:ASM} is applicable.
In this case, the Bregman distance of $F$ is given by
\begin{equation*}
D_F (u, v) = \frac{1}{2} a(u-v, u-v), \quad u, v \in H.
\end{equation*}
We equip $H$ with the energy norm $\| u \|_A = \sqrt{a(u, u)}$.
Then \cref{Ass:sharp} is true with $p = 2$ and $\mu_K = 1$ for all bounded and convex $K \subseteq H$.

In what follows, let an index $k$ run from $1$ to $N$.
Similarly to~\cref{decomposition}, we assume that $H$ admits a decomposition
\begin{equation*}
H = \sumk R_k^* H_k,
\end{equation*}
where $H_k$ is a Hilbert space and $R_k^*$:~$H_k \rightarrow H$ is a bounded linear operator with the surjective adjoint.
We set $d_k$ in~\cref{local_general} by
\begin{equation*}
d_k (w_k, v) = \frac{1}{2} \tilde{a}_k(w_k, w_k ), \quad w_k \in H_k,
\end{equation*}
for some continuous and SPD bilinear form $\tilde{a}_k (\cdot , \cdot)$ on $H_k$.
Note that the above definition of $d_k (w_k , v)$ is independent of $v$, so that we may simply write $d_k (w_k) = d_k (w_k , v)$ for $w_k \in V_k$ and $v \in V$.
In this setting, \cref{Ass:stable} with $q = 2$ is reduced to the following.

\begin{assumption}
\label{Ass:stable_linear}
There exists a constant $C_0 > 0$ which satisfies the following: for any $w \in H$, there exists $w_k \in H_k$, $1 \leq k \leq N$ such that
\begin{equation*}
w = \sumk \Rw
\end{equation*}
and
\begin{equation}
\label{stable_linear}
\sumk \tilde{a}_k ( w_k, w_k) \leq C_0^2 \| w \|_A^2.
\end{equation}
\end{assumption}

Compared to \cref{Ass:stable}, the dependency on the subset $K$ is dropped in \cref{Ass:stable_linear} since all terms in~\cref{stable_linear} are 2-homogeneous.
Then \cref{Ass:stable_linear} exactly agrees with~\cite[Assumption~2.2]{TW:2005}.
The following assumption is what \cref{Ass:convex} is reduced to.

\begin{assumption}
\label{Ass:convex_linear}
There exists a constant $\tau_0 > 0$ which satisfies the following:
for any $w_k \in H_k$, $1 \leq k \leq N$, and $\tau \in (0, \tau_0 ]$, we have
\begin{equation*}
a\left( \sumk \Rw, \sumk \Rw \right) \leq \frac{1}{\tau} \sumk a (\Rw, \Rw ).
\end{equation*}
\end{assumption}

We recall that~\cite[Assumption~2.3]{TW:2005}, also known as strengthened Cauchy--Schwarz inequalities on spaces $\{ H_k \}$, is written as follows:
there exists constants $\epsilon_{ij} \in [0, 1]$, $1 \leq i, j \leq N$, such that
\begin{equation}
\label{CS}
|a ( R_i^* w_i, R_j^* w_j ) | \leq \epsilon_{ij} a ( R_i^* w_i , R_i^* w_i )^{1/2} a (R_j^* w_j , R_j^* w_j )^{1/2}, \quad w_i \in H_i, \gap w_j \in H_j .
\end{equation}
Suppose that~\cref{CS} holds.
Then by the same argument as~\cite[Lemma~2.6]{TW:2005}, for $w_k \in H_k$ we have
\begin{equation*} \begin{split}
a\left( \sumk \Rw, \sumk \Rw \right) &= \sum_{i=1}^N \sum_{j=1}^N a ( R_i^* w_i, R_j^* w_j ) \\
&\leq \sum_{i=1}^N \sum_{j=1}^N \epsilon_{ij} a ( R_i^* w_i , R_i^* w_i )^{1/2} a (R_j^* w_j , R_j^* w_j )^{1/2} \\
&\leq \rho (\mathcal{E}) \sumk a ( \Rw, \Rw ),
\end{split} \end{equation*}
where $\rho (\mathcal{E} )$ is the spectral radius of the matrix $\mathcal{E} = [ \epsilon_{ij} ]_{i,j=1}^{N}$.
Therefore, $\tau_0 = 1 / \rho (\mathcal{E})$ satisfies \cref{Ass:convex_linear}.
In a trivial case of~\cref{CS} when $\epsilon_{ij} = 1$ for all $i$ and $j$, we have $\rho (\mathcal{E}) = N$ and it agrees with the trivial case $\tau_0 = 1/N$ of \cref{Ass:convex_linear} noted in the previous section.
In this sense, we can say that \cref{Ass:convex} is a generalization of~\cite[Assumption~2.3]{TW:2005}.
Finally, we consider a reduced version of \cref{Ass:local}.

\begin{assumption}
\label{Ass:local_linear}
There exists a constant $\omega_0 > 0$ which satisfies the following: for any $w_k \in H_k$, $1 \leq k \leq N$, we have
\begin{equation*}
a(R_k^* w_k, R_k^* w_k ) \leq \omega_0 \tilde{a}_k (w_k, w_k ).
\end{equation*}
\end{assumption}

\Cref{Ass:local_linear} has the same form as~\cite[Assumption~2.4]{TW:2005}.
In summary, \cref{Ass:stable,Ass:convex,Ass:local} can be regarded as generalizations of the three assumptions in the abstract convergence theory of Schwarz methods for linear elliptic problems presented in~\cite{TW:2005}.

Next, we claim that \cref{Lem:ASM} is a direct generalization of the well-known \textit{additive Schwarz lemma}~(see~\cite[Lemma~2.5]{TW:2005}), which plays a key role in the convergence analysis of Schwarz methods for linear problems.
Let
\begin{equation*}
\tilde{a}_k ( w_k, w_k ) = \langle \tilde{A}_k w_k, w_k \rangle, \quad w_k \in V_k,
\end{equation*}
for some continuous and SPD linear operator $\tilde{A}_k$:~$H_k \rightarrow H_k^*$.
We readily obtain
\begin{equation*}
d_k^* (p_k) = \frac{1}{2} \langle p_k, \tilde{A}_k^{-1} p_k \rangle, \quad p_k \in V_k^*.
\end{equation*}
For fixed $v \in V$, we write $\M (u) = \M (u, v)$, where $\M (u,v)$ was defined in~\cref{M}.
That is,
\begin{equation*} \begin{split}
\M (u) &= \tau \inf \left\{ \sumk \frac{\omega}{2} \tilde{a}_k (w_k, w_k ) : u-v = \tau \sumk \Rw, \gap w_k \in V_k \right\} \\
&= \frac{\omega}{2\tau} \inf \left\{ \sumk \tilde{a}_k (w_k, w_k ) : u-v = \sumk \Rw, \gap w_k \in V_k \right\}.
\end{split}\end{equation*}
By the same argument as~\cref{ASM_preconditioner}, we have
\begin{equation}
\label{ASM_preconditioner1}
\M^* (p) = \frac{\tau}{2\omega} \left< p,  \left( \sumk R_k^* \tilde{A}_k^{-1} R_k \right) p \right> + \left< p, v \right> , \quad p \in V^*.
\end{equation}
Dualizing~\cref{ASM_preconditioner1} yields
\begin{equation}
\label{ASM_preconditioner2}
\M(u) = \frac{\omega}{2 \tau} \left< \left( \sumk R_k^* \tilde{A}_k^{-1} R_k \right)^{-1}(u-v), u-v \right>, \quad u \in V.
\end{equation}
That is, the functional $\M$ is in fact a scaled quadratic form induced by the \textit{additive Schwarz preconditioner} $M$:~$V \rightarrow V^*$, which is defined by
\begin{equation*}
M = \left( \sumk R_k^* \tilde{A}_k^{-1} R_k \right)^{-1}.
\end{equation*}
Consequently, \cref{Lem:ASM,ASM_preconditioner2} imply that \cref{Alg:ASM} for~\cref{model_linear} is the preconditioned Richardson method
\begin{equation*}
\unn = \un - \frac{\tau}{\omega} M^{-1} (A\un - f), \quad n \geq 0.
\end{equation*}
Let $P_{\mathrm{ad}} = M^{-1} A$ be the additive operator introduced in~\cite[Section~2.2]{TW:2005}.
Then we have
\begin{equation*}
a(P_{\mathrm{ad}}^{-1}u, u) = \left< Mu, u \right> = \inf \left\{ \sumk \tilde{a}_k (u_k, u_k ): u = \sumk R_k^* u_k, \gap u_k \in V_k \right\},
\end{equation*}
which is the conclusion of the classical additive Schwarz lemma.
In this sense, we call \cref{Lem:ASM} the generalized additive Schwarz lemma. 
 
Under \cref{Ass:stable_linear,Ass:convex_linear,Ass:local_linear}, one can easily prove using~\cref{ASM_preconditioner2} that
\begin{equation*}
\frac{\tau_0 }{\omega_0} \| w \|_A^2 \leq \left< Mw, w \right> \leq C_0^2 \| w \|_A^2.
\end{equation*}
Therefore, the condition number of the preconditioned operator $P_{\mathrm{ad}}$ is bounded by $\omega_0 C_0^2 / \tau_0$.
This bound agrees with~\cite[Theorem~2.7]{TW:2005}.
Moreover, it agrees with~\cref{kASM} in the case $\tau = \tau_0$ and $\omega = \omega_0$.
Therefore, the additive Schwarz condition number $\kASM$ introduced in~\cref{Sec:ASM} generalizes the condition number of $P_{\mathrm{ad}}$.

\section{Overlapping domain decomposition}
\label{Sec:DD}
In this section, we present overlapping domain decomposition settings for finite element spaces that will be used in this paper.
In the remainder of the paper, let $\Omega$ be a bounded polygonal domain in $\mathbb{R}^d$, where $d$ is a positive integer.
The notation $A \lesssim B$ means that there exists a constant $c > 0$ such that $A \leq c B$, where $c$ is independent of the parameters $H$, $h$, and $\delta$ which are related to the geometry of domain decomposition and will be defined later.
We also write $A \approx B$ if $A \lesssim B$ and $B \lesssim A$.

As a coarse mesh, let $\T_H$ be a quasi-uniform triangulation of $\Omega$ with $H$ the maximal element diameter.
We refine the coarse mesh $\T_H$ to obtain a quasi-uniform triangulation $\T_h$ with $h<H$, which plays a role of a fine mesh.
Let $S_H (\Omega) \subset W_0^{1, \infty} (\Omega) $ and $S_h (\Omega) \subset W_0^{1, \infty} (\Omega)$ be the continuous, piecewise linear finite element spaces on $\T_H$ and $\T_h$ with the homogeneous essential boundary condition, respectively.
For sufficiently smooth functions, the nodal interpolation operators $I_H$ and $I_h$ onto $S_H (\Omega)$ and $S_h (\Omega)$, respectively, are well-defined.

We decompose $\Omega$ into $N$ nonoverlapping subdomains $\{ \Omega_k \}_{k=1}^N$ such that each $\Omega_k$ is the union of some coarse elements in $\T_H$, and the number of coarse elements  consisting of $\Omega_k$ is uniformly bounded.
For each $\Omega_k$, we make a larger region $\Omega_k'$ by adding layers of fine elements with the width $\delta$.
We define $S_h (\Omega_k ') \subset W_0^{1, \infty}(\Omega_k')$ as the continuous, piecewise linear finite element space on the $\T_h$-elements in $\Omega_k'$ with the homogeneous essential boundary condition.

In the additive Schwarz framework presented in \cref{Sec:ASM}, we set
\begin{equation}
\label{Vk}
V = S_h (\Omega) \quad \textrm{and}\quad V_k = S_h (\Omega_k'), \quad 1 \leq k \leq N.
\end{equation} 
We also define
\begin{equation*}
V_0 = S_H (\Omega).
\end{equation*}
We take $R_k^*$:~$V_k \rightarrow V$ as the natural extension operator for $1 \leq k \leq N$, and $R_0^*$:~$V_0 \rightarrow V$ as the natural interpolation operator.
Then it is clear that
\begin{equation}
\label{1L}
V = \sumk R_k^* V_k
\end{equation}
and
\begin{equation}
\label{2L}
V = R_0^* V_0 + \sumk R_k^* V_k.
\end{equation}
We say that an additive Schwarz method is said to be \textit{one-level} if it uses the space decomposition~\cref{1L}, while it is called \textit{two-level} if it uses~\cref{2L}.

\subsection{Coloring technique}
In \cref{Sec:ASM}, we noted that a constant $\tau_0$ in \cref{Ass:convex} larger than $1/N$ can be found by the coloring technique.
Now, we explain the details.
In the proposed method, we say that two spaces $V_i$ and $V_j$ are \textit{of the same color} if
\begin{multline}
\label{color}
E (v + R_i^* w_i + R_j^* w_j ) - E(v) \\
= \left( E(v + R_i^* w_i) - E(v) \right) + \left( E(v + R_j^* w_j) - E(v) \right), \quad v \in V, \gap w_i \in V_i, \gap w_j \in V_j.
\end{multline}
Inductively, one can prove the following: if $V_{k_1},\dots, V_{k_m}$ are of the same color, then we have
\begin{equation}
\label{color2}
E \left( v + \sum_{i=1}^{m} R_{k_i}^* w_{k_i} \right) - E(v) = \sum_{i=1}^{m} \left( E ( v + R_{k_i}^* w_{k_i} ) - E(v) \right), \quad v \in V, \gap w_{k_i} \in V_{k_i}.
\end{equation}
Assume that the local spaces $\{ V_k \}_{k=1}^N$ are classified into $N_c$ colors according to~\cref{color} for some $N_c \leq N$.
Let $\mathcal{I}_j$, $1 \leq j \leq N_c$ be the set of the indices $k$ such that $V_k$ is of the color $j$.
Then for $\tau \in (0, 1/N_c ]$, $v \in V$, and $w_k \in V_k$, we have
\begin{equation*} \begin{split}
(1-&\tau N) E(v) + \tau \sumk E(v + \Rw ) - E \left( v + \tau \sumk \Rw \right) \\
&\stackrel{\cref{color2}}{=} \tau  \sum_{j=1}^{N_c} \left( E \left( v + \sum_{k \in \mathcal{I}_j} w_k \right) - E(v) \right) - \left( E \left( v + \tau \sum_{j=1}^{N_c} \sum_{k \in \mathcal{I}_j} \Rw \right) - E(v) \right) \\
&= \tau  \sum_{j=1}^{N_c} E \left( v + \sum_{k \in \mathcal{I}_j} w_k \right) + (1 - \tau N_c ) E(v) - E \left( v + \tau \sum_{j=1}^{N_c} \sum_{k \in \mathcal{I}_j} \Rw \right)  \\
&\geq 0,
\end{split} \end{equation*}
where the last inequality is due to the convexity of $E$ and $1 - \tau N_c \geq 0$.
Therefore, \cref{Ass:convex} is true with $\tau_0 = 1/N_c$.

In most applications, $E(u)$ has the integral structure and naturally satisfies~\cref{color}.
As a descriptive example, let $V$ and $V_k$ be given by~\cref{Vk} and
\begin{equation*}
E(u) = \frac{1}{s} \intO |\nabla u|^s \,dx - \left< f, u \right>
\end{equation*}
for $s > 1$ and $f \in V^*$.
Then it is obvious that $V_i$ and $V_j$ are of the same color if $\bar{\Omega}_i'  \cap \bar{\Omega}_j' = \emptyset$.
Hence, for suitable overlap parameter $\delta$, we have
\begin{equation*}
N_c \leq \begin{cases} 2 & \textrm{ if } d = 1, \\ 4 & \textrm{ if } d = 2,\\ 8 & \textrm{ if } d = 3. \end{cases}
\end{equation*}
For two-level methods, we have $\tau_0 = 1/(N_c + 1)$ because of the coarse space $V_0$.
In summary, we have
\begin{equation}
\label{tau0}
\tau_0 = \begin{cases} \frac{1}{N_c} & \textrm{ for one-level}~\cref{1L}, \\ \frac{1}{N_c + 1} & \textrm{ for two-level}~\cref{2L}. \end{cases}
\end{equation}

We conclude the section by observing a special case when $V = H$ is a Hilbert space and
\begin{equation*}
E(u) = \frac{1}{2} \left< Au, u \right> - \left< f, u \right>
\end{equation*}
for a continuous, symmetric, positive definite linear operator $A$:~$H \rightarrow H^*$ and $f \in H^*$.
Then~\cref{color} reduces to
\begin{equation*}
\left< A R_i^* w_i , R_j^* w_j \right> = 0, \quad w_i \in V_i, \gap w_j \in V_j,
\end{equation*}
which agrees with~\cite[Section~2.5.1]{TW:2005}.
In this sense, the proposed coloring technique generalizes the theory developed in~\cite{TW:2005}.

\subsection{One-level domain decomposition}
First, we consider the one-level domain decomposition~\cref{1L}.
By~\cite[Lemma~3.4]{TW:2005}, we can choose a continuous and piecewise linear partition of unity $\{ \theta_k \}_{k=1}^N$ for $\Omega$ subordinate to the covering $\{ \Omega_k' \}_{k=1}^N$ satisfying~\cite[equations~(3.2) and~(3.3)]{TW:2005}.
Invoking~\cite[Lemmas~3.4 and~3.9]{TW:2005}, the following lemma is straightforward under the space decomposition~\cref{1L}.

\begin{lemma}
\label{Lem:1L}
Assume that the space $V$ is decomposed according to~\cref{1L}. 
For $w \in V$, we choose $w_k \in V_k$, $1 \leq k \leq N$ such that
\begin{equation}
\label{1L_decomp}
R_k^* w_k =  I_h (\theta_k w ).
\end{equation}
Then for $s \geq 1$, we have $w = \sumk \Rw$ and
\begin{equation*}
\sumk \| \Rw \|_{W^{1,s} ( \Omega )} \lesssim C_{N_c} \left( 1 + \frac{1}{\delta} \right) \| w \|_{W^{1,s} ( \Omega )},
\end{equation*}
where $C_{N_c}$ is a positive constant depending on the number of colors $N_c$ only.
\end{lemma}

\subsection{Two-level domain decomposition}
There are several results on stable decompositions for the two-level domain decomposition~\cref{2L} which are counterparts to \cref{Lem:1L}.
If we choose a coarse component $w_0 \in V_0$ of $w\in V$ by the $L^2$-projection technique, we obtain the following estimate~\cite[Lemma~4.1]{TX:2002}.

\begin{lemma}
\label{Lem:2L}
Assume that the space $V$ is decomposed according to~\cref{2L}. 
For $w \in V$, let $w_0 \in V_0$ such that $R_0^* w_0$ is the $L^2$-projection of $w$, i.e.,
\begin{subequations}
\label{2L_decomp}
\begin{equation}
\intO (R_0^* w_0 - w) R_0^* \phi_0 \,dx  = 0, \quad \phi_0 \in V_0.
\end{equation}
Then we choose $w_k \in V_k$, $1 \leq k \leq N$ such that
\begin{equation}
R_k^* w_k =  I_h (\theta_k (w - R_0^* w_0 ) ).
\end{equation}
\end{subequations}
For $s \geq 1$, we have $w = R_0^* w_0 + \sumk \Rw$ and
\begin{equation*}
\|R_0^* w_0 \|_{W^{1,s} (\Omega)} +  \sumk \| \Rw \|_{W^{1,s} ( \Omega )} \lesssim C_{N_c} \left( 1 + \left(\frac{H}{\delta} \right)^{\frac{s-1}{s}} \right) \| w \|_{W^{1,s} ( \Omega )},
\end{equation*}
where $C_{N_c}$ is a positive constant depending on the number of colors $N_c$ only.
\end{lemma}

In applications to nonsmooth optimization problems, we may need a decomposition different from \cref{Lem:2L} in order to satisfy~\cref{stable_nonsmooth}~\cite{BTW:2003,Tai:2003}.
Let $I_H^{\ominus}$:~$S_h (\Omega) \rightarrow S_H (\Omega)$ be the nonlinear interpolation operator defined in~\cite[Section~4]{Tai:2003}.
One may refer to~\cite{Badea:2006,Tai:2003} for some useful estimates related to $I_H^{\ominus}$.
Then we have the following estimate on a decomposition using $I_H^{\ominus}$~\cite[Proposition~4.1]{Badea:2006}.

\begin{lemma}
\label{Lem:2L_nonsmooth}
Assume that the space $V$ is decomposed according to~\cref{2L}. 
For $w \in V$, we define $w_0 \in V_0$ by
\begin{subequations} \label{2L_decomp_nonsmooth}
\begin{equation}
R_0^* w_0 = I_H^{\ominus}\left( \max (0, w) \right) -  I_H^{\ominus}\left( \max (0, -w) \right).
\end{equation}
Then we choose $w_k \in V_k$, $1 \leq k \leq N$ such that
\begin{equation}
R_k^* w_k =  I_h (\theta_k (w - R_0^* w_0 ) ).
\end{equation}
\end{subequations}
For $s \geq 1$, we have $w = R_0^* w_0 + \sumk \Rw$ and
\begin{equation*}
\|R_0^* w_0 \|_{W^{1,s} (\Omega)} +  \sumk \| \Rw \|_{W^{1,s} ( \Omega )} \lesssim C_{N_c} C_{d, s}(H,h) \left( 1 + \frac{H}{\delta} \right) \| w \|_{W^{1,s} ( \Omega )},
\end{equation*}
where $C_{N_c}$ is a positive constant depending on the number of colors $N_c$ only and
\begin{equation*}
C_{d,s} (H,h) = \begin{cases} 1 & \textrm{ if }\gap d=s=1 \textrm{ or } 1 \leq d < s, \\ \left( 1 + \log \frac{H}{h} \right)^{\frac{s-1}{s}} & \textrm{ if }\gap 1<d=s, \\ \left( \frac{H}{h} \right)^{\frac{d-s}{s}} & \textrm{ if }\gap 1 \leq s < d. \end{cases}
\end{equation*}
\end{lemma}

\section{Applications}
\label{Sec:Applications}
In this section, we present various applications of the proposed abstract convergence theory for additive Schwarz methods.
The proposed theory covers many interesting convex optimization problems: nonlinear elliptic problems, nonsmooth problems, and nonsharp problems.
It also gives a unified analysis with some other decomposition methods such as block coordinate descent methods and constraint decomposition methods.
The proposed theory can adopt stable decomposition estimates for Schwarz methods presented in existing works without modification.
This makes the convergence analysis of additive Schwarz methods easy and gives an equivalent or even better estimate for the convergence rate compared to existing works.

\subsection{Nonlinear elliptic problems}
We present applications of the proposed additive Schwarz method to some nonlinear elliptic partial differential equations on $\Omega$.
We consider the minimization problem
\begin{equation}
\label{s_Lap}
\min_{u \in W_0^{1, s}(\Omega )} \left\{ E(u) := \frac{1}{s} \intO |\nabla u|^s \,dx - \left< f, u \right> \right\}
\end{equation}
for some $s > 1$ such that $s \neq 2$ and $f \in W^{-1, s^*} (\Omega)$, where $s^*$ is the H{\"o}lder conjugate of $s$, i.e., $\frac{1}{s} + \frac{1}{s^*} = 1$.
The unique solution of~\cref{s_Lap} is characterized by a solution of the well-known $s$-Laplacian equation
\begin{align*} 
- \div \left( |\nabla u|^{s-2} \nabla u \right) &= f \quad\textrm{ in } \Omega, \\
u &= 0 \quad\textrm{ on } \partial \Omega .
\end{align*}
It is well-known that there exist two positive constants $\alpha_s$ and $\beta_s$ such that for any $u, v \in W_0^{1, s}(\Omega)$, we have
\begin{subequations}
\label{s>2}
\begin{equation}
\left< E'(u) - E'(v) , u - v \right> \geq \alpha_s \| u - v \|_{W^{1,s}(\Omega)}^s,
\end{equation}
\begin{equation}
\| E'(u)  - E'(v) \|_{W^{-1,s^*}(\Omega)} \leq \beta_s \left( \| u \|_{W^{1,s}(\Omega)} + \| v \|_{W^{1,s}(\Omega)} \right)^{s-2} \| u - v \|_{W^{1,s}(\Omega)}
\end{equation}
\end{subequations}
if $s > 2$ and
\begin{subequations}
\label{s<2}
\begin{equation}
\left< E'(u) - E'(v) , u - v \right> \geq \alpha_s \frac{\| u - v \|_{W^{1,s}(\Omega)}^2}{(\| u \|_{W^{1,s}(\Omega)} + \| v \|_{W^{1,s}(\Omega)})^{2-s}},
\end{equation}
\begin{equation}
\| E'(u)  - E'(v) \|_{W^{-1,s^*}(\Omega)} \leq \beta_s \| u - v \|_{W^{1,s}(\Omega)}^{s-1}
\end{equation}
\end{subequations}
if $1 <s <2$, where $\| \cdot \|_{W^{-1,s^*}(\Omega)}$ is the dual norm of $\| \cdot \|_{W^{1,s}(\Omega)}$; see~\cite{Ciarlet:2002}.

A conforming finite element approximation of~\cref{s_Lap} using $S_h (\Omega) \subset W_0^{1,s} (\Omega)$ is given by
\begin{equation}
\label{d_s_Lap}
\min_{u \in S_h(\Omega )} \left\{ E_h(u) := \frac{1}{s} \intO |\nabla u|^s \,dx - \left< f, u \right> \right\}.
\end{equation}
Clearly,~\cref{d_s_Lap} is an instance of~\cref{model_gradient} with
\begin{equation*}
V = S_h(\Omega ), \quad F(u) = \frac{1}{s} \intO |\nabla u|^s \,dx - \left< f, u \right>, \quad G(u) = 0.
\end{equation*}
One can show that $E_h$ is coercive without difficulty~\cite{TX:2002}.

We take any bounded and convex subset $K$ of $V$ and define $M_K > 0$ by
\begin{equation*}
M_K = \sup_{u \in K} \| u \|_{W^{1,s}(\Omega)} < \infty.
\end{equation*}
We choose $u,v \in K$ arbitrarily.
Note that $v + t(u-v) \in K$ for any $t \in [0,1]$.
It follows by the fundamental theorem of calculus that
\begin{equation} \begin{split}
\label{FTC}
D_F (u,v) &= \int_0^1 \left< F'(v + t(u-v)), u-v \right> \,dt - \left< F'(v) , u -v \right> \\
&= \int_0^1 \frac{1}{t} \left< F'(v + t(u-v)) - F'(v) , t(u-v) \right> \,dt.
\end{split} \end{equation}
If $s > 2$, combining~\cref{FTC,s>2}, we have
\begin{subequations} \label{d_s>2}
\begin{equation}
D_F(u,v) \geq \frac{\alpha_s}{s} \| u - v \|_{W^{1,s}(\Omega)}^s,
\end{equation}
\begin{equation}
D_F(u,v) \leq \frac{\beta_s (2M_K)^{s-2}}{2} \| u-v \|_{W^{1,s}(\Omega)}^2 .
\end{equation}
\end{subequations}
Similarly, if $1 <s <2$, then we obtain
\begin{subequations} \label{d_s<2}
\begin{equation}
D_F(u,v) \geq \frac{\alpha_s}{(2M_K)^{2-s}} \| u - v \|_{W^{1,s}(\Omega)}^2,
\end{equation}
\begin{equation}
D_F(u,v) \leq \frac{\beta_s}{s} \| u-v \|_{W^{1,s}(\Omega)}^s 
\end{equation}
\end{subequations}
by using~\cref{FTC,s<2}.

Suppose that we want to solve~\cref{d_s_Lap} by \cref{Alg:ASM}.
We should verify \cref{Ass:stable,Ass:convex,Ass:local} to ensure the convergence, and \cref{Ass:sharp} if possible.
If we use the domain decompositions given by either~\cref{1L} or~\cref{2L}, then \cref{Ass:convex} is straightforward with~\cref{tau0}.
\Cref{Ass:local} holds with $\omega_0 = 1$ in the case of the exact local solvers.
Moreover, invoking \cref{Prop:uniform} to either~\cref{d_s>2} or~\cref{d_s<2} implies \Cref{Ass:sharp}.

Next, we show that \cref{Ass:stable} is valid for the two domain decompositions~\cref{1L,2L}.
Take any bounded and convex subset $K$ of $V$ and let $u, v \in V$ with $w = u-v$.
For the one-level domain decomposition~\cref{1L}, we set $w_k \in V_k$, $1 \leq k \leq N$ as~\cref{1L_decomp}.
For the two-level case domain decomposition~\cref{2L}, we set $w_k \in V_k$, $0 \leq k \leq N$ as~\cref{2L_decomp}.
In the case of $s > 2$ and the one-level domain decomposition, by~\cref{d_s>2,Lem:1L} we have
\begin{equation*} \begin{split}
\sumk d_k (w_k, v) &\lesssim \sumk \| \Rw \|_{W^{1,s}(\Omega)}^2 \\
&\lesssim \left( 1 + \frac{1}{\delta^2} \right) \| w \|_{W^{1,s}(\Omega)}^2.
\end{split} \end{equation*}
Similarly, the following estimate can be obtained using~\cref{d_s>2,Lem:2L} in the two-level domain decomposition:
\begin{equation*}
d_0 (w_0, v) + \sumk d_k (w_k, v) \lesssim \left( 1 + \left(\frac{H}{\delta}\right)^{\frac{2(s-1)}{s}} \right) \| w \|_{W^{1,s}(\Omega)}^2.
\end{equation*}
Therefore, \cref{Ass:stable} is satisfied if $s > 2 $.
Results corresponding to the case $1<s<2$ can be obtained by the same argument.

In summary, by \cref{Thm:ASM_uniform} we have the following convergence results of \cref{Alg:ASM} for~\cref{d_s_Lap}.

\begin{theorem}
\label{Thm:s_Lap}
In \cref{Alg:ASM} for~\cref{d_s_Lap}, suppose that we set $\tau_0$ as~\cref{tau0} and use the exact local solvers.
If $E(u^{(0)}) - E(u^*)$ is small enough, then we have
\begin{align*}
E_h (\un) - E_h (u^*) &\lesssim \frac{1 + 1/\delta^q}{(n+1)^{\frac{p(q-1)}{p-q}}} \hspace{0.8cm} \textrm{for one-level}, \\
E_h (\un) - E_h (u^*) &\lesssim \frac{1 + (H/\delta)^{\frac{q(s-1)}{s}}}{(n+1)^{\frac{p(q-1)}{p-q}}} \quad \textrm{for two-level},
\end{align*}
where
\begin{align*}
p=s, \gap q=2 \quad &\textrm{ if }\gap s > 2, \\
p=2, \gap q=s \quad &\textrm{ if }\gap 1 < s < 2.
\end{align*}
\end{theorem}

A remarkable property of~\cref{Alg:ASM} for~\cref{d_s_Lap} is that the convergence of the method is not significantly affected by the initial energy error $E_h(u^{(0)}) - E_h(u^*)$, even though the asymptotic convergence rate is only sublinear.
In~\cref{Thm:ASM_uniform}, we showed that the energy error decays linearly if it is sufficiently large.
Therefore, the number of iterations required to meet a prescribed stop condition does not become too large even if $E_h(u^{(0)}) - E_h(u^*)$ is very big.
We note that a similar discussion was done in~\cite{Beck:2015}.

Finally, we compare the above estimates with existing ones.
In~\cite{TX:2002}, it was proven that \cref{Alg:ASM} applied to~\cref{d_s_Lap} has the $O(1/n^{\frac{q(q-1)}{(p-q)(p+q-1)}})$ convergence rate of the energy error.
More recently, the $O(1/n^{\frac{q-1}{p-q}})$ convergence of the energy error was shown in~\cite{Badea:2006,BK:2012}.
Since
\begin{equation*}
\frac{q(q-1)}{(p-q)(p+q-1)} < \frac{q-1}{p-q} < \frac{p(q-1)}{p-q}
\end{equation*}
for $1<q<p$, we conclude that \cref{Thm:s_Lap} is sharper than the existing results mentioned above.

\subsection{Nonsmooth problems}
We deal with the problems of the form~\cref{model_gradient} with the nonzero nonsmooth parts, i.e., $G \neq 0$.
Suppose that $G$ satisfies the following assumption, which was previously stated in~\cite{Badea:2010,BK:2012}.

\begin{assumption}
\label{Ass:nonsmooth}
Let $\mathcal{N}_h$ be the set of vertices in the triangulation $\T_h$.
Then $G$ can be expressed as
\begin{equation*}
G(u) = \sum_{x \in \mathcal{N}_h} s_x (h) \phi (u(x))
\end{equation*}
for some convex functions $\phi$:~$\mathbb{R} \rightarrow \overline{\mathbb{R}}$ and $s_x(h) \geq 0$.
\end{assumption}

\Cref{Ass:nonsmooth} means that $G$ is the sum of pointwisely defined convex functions.
Various examples satisfying \cref{Ass:nonsmooth} can be found in~\cite{BK:2012}.
Here, we consider the following $L^1$-regularized obstacle problem~\cite{TSFO:2015}:
\begin{equation}
\label{L1_obstacle}
\min_{u \in H_0^1 (\Omega)} \left\{ E(u) := \frac{1}{2} \intO |\nabla u|^2 \,dx - \left< f, u \right> + \lambda \intO  |u| \,dx \right\},
\end{equation}
where $f \in H^{-1}(\Omega)$ and $\lambda > 0$.
Note that~\cref{L1_obstacle} has an equivalent variational inequality of the second kind of the form~\cref{optimality}.
We show that the additive Schwarz method for variational inequalities of the second kind proposed in~\cite{Badea:2010} can be represented in our framework.

A finite element approximation of~\cref{L1_obstacle} using $S_h (\Omega) \subset H_0^1 (\Omega)$ is written as
\begin{equation}
\label{d_L1_obstacle}
\min_{u \in S_h (\Omega)} \left\{ E_h (u) := \frac{1}{2} \intO |\nabla u|^2 \,dx - \left< f, u \right> + G_h (u) \right\},
\end{equation}
where $G_h(u)$ is the numerical integration of $\lambda |u|$ using the piecewise linear approximation, i.e.,
\begin{equation*}
G_h (u) = \lambda \intO I_h |u| \, dx.
\end{equation*}
It is clear that $G_h$ satisfies \cref{Ass:nonsmooth}.
We focus on the fact that~\cref{d_L1_obstacle} is an instance of~\cref{model_gradient} with
\begin{equation*}
V = S_h (\Omega) , \quad F(u) = \frac{1}{2}\intO |\nabla u|^2 \,dx - \left< f, u \right>, \quad G(u) = G_h (u).
\end{equation*}

We analyze the convergence behavior of \cref{Alg:ASM} applied to~\cref{d_L1_obstacle} with the space decompositions~\cref{1L,2L}.
Assume that we use the exact local solvers.
Then \cref{Ass:convex,Ass:local} are trivially satisfied with~\cref{tau0} and $\omega_0 = 1$, respectively.
\Cref{Ass:sharp} with $p = 2$ and $\mu \approx 1$ is verified by
\begin{equation*}
\frac{1}{2} \| u - v \|_{H^1 (\Omega)}^2 \approx \frac{1}{2} | u -v |_{H^1 (\Omega)}^2 = D_F (u, v) , \quad u,v \in V,
\end{equation*}
followed by an application of~\cref{Prop:uniform}, where we used the Poincar\'{e}--Friedrichs inequality for $H_0^1 (\Omega)$ in $\approx$.

Now, we prove \cref{Ass:stable} for two domain decompositions~\cref{1L,2L}.
Take any $u, v \in V$ and let $w = u-v$.
For the one-level domain decomposition~\cref{1L}, we set $w_k \in V_k$, $1\leq k \leq N$ as~\cref{1L_decomp}.
For the two-level case, we set $w_k \in V_k$, $0 \leq k \leq N$ as~\cref{2L_decomp_nonsmooth}.
Then using \cref{Lem:1L} and \cref{Ass:nonsmooth}, it is straightforward to show that \cref{Ass:stable} holds in the one-level case with $q = 2$ and
\begin{equation*}
C_{0,K} \lesssim 1 + \frac{1}{\delta }
\end{equation*}
for all $K$; see~\cite[Proposition~5.1]{Badea:2010}.
In addition, by using \cref{Lem:2L_nonsmooth} and closely following~\cite[Proposition~5.2]{Badea:2010}, \cref{Ass:stable} is satisfied for the two-level domain decomposition with $q = 2$ and
\begin{equation*}
C_{0,K} \lesssim C_{d,2}(H,h) \left(1 + \frac{H}{\delta} \right)
\end{equation*}
for all $K$.

In conclusion, invoking \cref{Thm:ASM_uniform} yields the following convergence theorem for \cref{Alg:ASM}  applied to~\cref{d_L1_obstacle}.

\begin{theorem}
\label{Thm:nonsmooth}
In \cref{Alg:ASM} for~\cref{d_L1_obstacle}, suppose that we set $\tau_0$ as~\cref{tau0} and use the exact local solvers.
Then we have
\begin{equation*}
\resizebox{0.85\textwidth}{!}{$ \displaystyle E_h(\un ) - E_h(u^*) \leq \left( 1- \frac{1}{2} \min \left\{ \tau, \frac{C}{ 1+ 1/\delta^2 } \right\} \right)^n ( E_h(u^{(0)}) - E_h(u^*) )  \quad \textrm{for one-level},$}
\end{equation*}
\begin{equation*}
\resizebox{\textwidth}{!}{$ \displaystyle E_h(\un) - E_h(u^*) \leq  \left( 1- \frac{1}{2} \min \left\{ \tau, \frac{C}{C_{d,2}(H,h)^2(1+(H/\delta)^2)} \right\} \right)^n ( E_h(u^{(0)}) - E_h(u^*) ) \gap\textrm{for two-level},$}
\end{equation*}
where $C > 0$ is a generic constant independent of $H$, $h$, and $\delta$.
\end{theorem}

\Cref{Thm:nonsmooth} agrees with the existing results in~\cite{Badea:2010} in the sense that the linear convergence rate is dependent on the bounds for $C_{0,K}$.

We remark that constrained problems belong to the class of nonsmooth problems.
Indeed, for a nonempty, convex, and closed subset $K_h$ of $V = S_h (\Omega)$, the constrained minimization problem
\begin{equation*}
\min_{u \in K_h} F(u)
\end{equation*}
can be represented as an unconstrained and nonsmooth minimization problem
\begin{equation}
\label{d_obstacle}
\minu \left\{ F(u) + \chi_{K_h} (u) \right\},
\end{equation}
where the functional $\chi_{K_h}$ was defined in~\cref{chi}.
Therefore, additive Schwarz methods for constrained problems can be analyzed in the same way as above.
If we let $G = \chi_{K_h} $ in~\cref{model_gradient}, then \cref{Ass:nonsmooth} reduces to the following.

\begin{assumption}
\label{Ass:constraint}
Let $\theta \in S_h (\Omega)$ with $0 \leq \theta \leq 1$.
Then for $u ,v \in K_h$, we have $I_h (\theta u + (1-\theta)v) \in K_h$.
\end{assumption}

It is clear that \cref{Ass:constraint} holds when $K_h$ is defined in terms of pointwise constraints.
The same assumptions as \cref{Ass:constraint} were used in existing works~\cite{Badea:2006,BTW:2003} on obstacle problems.

\subsection{Absence of the sharpness}
The examples we had presented above satisfied \cref{Ass:sharp}.
Now, we provide an application of the proposed framework to a problem lacking the sharpness, i.e., not satisfying \cref{Ass:sharp}.
As a model problem, we consider the following:
\begin{equation}
\label{dual_TV}
\min_{\bu \in H_0 (\div ; \Omega)} \left\{ E(\bu) := \tF (\div \bu ) + \chi_K (\bu) \right\},
\end{equation}
where $\tF$:~$L^2 (\Omega) \rightarrow \mathbb{R}$ is a convex, Frech\'{e}t differentiable functional and $K$ is the subset of $H_0 (\div; \Omega)$ defined by
\begin{equation*}
K = \left\{ \bu \in H_0 (\div; \Omega) : |\bu | \leq 1 \gap\textrm{ a.e. in }\Omega \right\}.
\end{equation*}
We further assume that $\tF'$ is H{\"o}lder continuous with parameters $q-1 \in (0, 1]$ and $\tilde{L} > 0$, so that
\begin{equation}
\label{TV_Holder}
D_{\tF} (u, v) \leq \frac{\tilde{L}}{q} \|u-v \|_{L^2 (\Omega)}^q, \quad u,v \in L^2 (\Omega);
\end{equation} 
see~\cite[Lemma~2.1]{TX:2002}.
Problems of the form~\cref{dual_TV} are typical in mathematical imaging.
More precisely,~\cref{dual_TV} appears in Fenchel--Rockafellar dual problems of total variation regularized problems which are standard in mathematical imaging~\cite{HK:2004,LP:2019}.
Schwarz methods for~\cref{dual_TV} have been studied recently in~\cite{CTWY:2015,Park:2019}.

A discrete counterpart of~\cref{dual_TV} can be obtained by replacing $H_0 (\div; \Omega)$ by the lowest order Raviart--Thomas finite element space $\bS_h (\Omega)$~\cite{HHSVW:2019,LP:2019}:
\begin{equation}
\label{d_dual_TV}
\min_{\bu \in \bS_h (\Omega)} \left\{ E_h (\bu) := \tF(\div \bu ) + \chi_{K_h}(\bu) \right\}.
\end{equation}
In~\cref{d_dual_TV}, $K_h$ is the convex subset of $\bS_h (\Omega)$ given by
\begin{equation*}
K_h = \left\{ \bu \in \bS_h (\Omega) : \frac{1}{e} \int_{e} |\bu \cdot \mathbf{n}_e| \,dS \leq 1, \gap e\textrm{: interior faces of }\T_h \right\},
\end{equation*}
where $\mathbf{n}_e$ is the unit outer normal to $e$.
See, e.g.,~\cite{Oh:2013} for further properties of the space $\bS_h (\Omega)$.
We denote a solution of~\cref{d_dual_TV} by $\bu^*$.
We observe that~\cref{d_dual_TV} is of the form~\cref{model_gradient} with
\begin{equation*}
V = \bS_h, \quad F(\bu) = \tF (\div \bu), \quad G(\bu) = \chi_{K_h} (\bu).
\end{equation*}
The energy functional $E$ is coercive due to the $\chi_{K_h}$-term.
Because of the large null space of $\div$ operator,~\cref{dual_TV} does not satisfy \cref{Ass:sharp}.

Based on the overlapping domain decomposition $\{ \Omega_k \}$ introduced in \cref{Sec:DD}, we define
\begin{equation*}
V_k = \bS_h (\Omega_k'), \quad 1 \leq k \leq N,
\end{equation*}
where $\bS_h (\Omega_k')$ is the Raviart--Thomas finite element space on $\Omega_k'$ with the homogeneous essential boundary condition.
Then it satisfies that
\begin{equation}
\label{1L_TV}
V = \sumk R_k^* V_k,
\end{equation}
where $R_k^*$:~$V_k \rightarrow V$ is the natural embedding.

We investigate the convergence property of \cref{Alg:ASM} applied to~\cref{d_dual_TV} based on the space decomposition~\cref{1L_TV}.
If we use the exact local solvers, then \cref{Ass:convex,Ass:local} are trivial with $\tau_0 = 1/N_c$ and $\omega_0 = 1$.
In order to verify \cref{Ass:stable}, we choose any $\bu , \bv \in K_h$.
We define $\bw_k \in V_k$, $1 \leq k \leq N$ such that
\begin{equation*}
R_k^* \bw_k = \Pi_h (\theta_k (\bu - \bv)),
\end{equation*}
where $\Pi_h$ is the nodal interpolation operator onto $\bS_h (\Omega)$.
Then we clearly have $\bv + R_k^* \bw_k \in K_h$ and~\cref{stable_nonsmooth} holds.
Moreover, we get
\begin{equation*} \begin{split}
\sumk d_k ( \bw_k , \bv) &= \sumk D_{\tF \circ \div} (\bv + R_k^* \bw_k, \bv) \\
&\stackrel{\cref{TV_Holder}}{\leq} \frac{\tilde{L}}{q} \| \div R_k^* \bw_k \|_{L^2(\Omega)}^q \\
&\stackrel{\textrm{(a)}}{\lesssim} C_{N_c} \tilde{L} \left( 1 + \frac{1}{\delta^q} \right) \| \bu - \bv \|_{H(\div; \Omega)}^q,
\end{split} \end{equation*}
where $C_{N_c}$ is a positive constant depending on $N_c$, and (a) is due to~\cite[Proposition~4.1]{Park:2019}.
Hence, \cref{Ass:stable} also holds.

By \cref{Thm:ASM}, we have the following convergence theorem for \cref{Alg:ASM} applied to~\cref{d_dual_TV}.

\begin{theorem}
\label{Thm:dual_TV}
In \cref{Alg:ASM} for~\cref{d_dual_TV} with the space decomposition~\cref{1L_TV}, suppose that we set $\tau_0 = 1/N_c$ and use the exact local solvers.
We also assume that \cref{TV_Holder} holds.
If $E(u^{(0)}) - E(u^*)$ is small enough, then we have
\begin{equation*}
E_h(\un ) - E_h(u^*)
\lesssim \frac{1 + 1/\delta^q}{(n+1)^{q-1}}.
\end{equation*}
\end{theorem}

Similarly to the case of~\cref{d_s_Lap}, one can observe from \cref{Thm:ASM_uniform} that the initial energy error $E_h (u^{(0)}) - E_h (u^*)$ does not affect the number of required iterations much.

Compared to the analysis in~\cite{Park:2019}, the result presented in \cref{Thm:dual_TV} only requires the H{\"o}lder continuity of $\tF'$, while~\cite{Park:2019} requires much stronger conditions: the Lipschitz continuity of $\tF'$ and the strong convexity of $\tF$.

\subsection{Inexact local solvers}
We present two notable instances of the proposed method with inexact local solvers: block coordinate descent methods and constraint decomposition methods.

Block coordinate descent methods are popular in convex optimization and there is a vast literature about them; for example, see~\cite{BT:2013,FR:2015,Tseng:2001}.
Here, we show that parallel block coordinate descent methods are instances of \cref{Alg:ASM} with inexact local solvers.
In \cref{Alg:ASM}, we set
\begin{equation*}
V = \prod_{k=1}^N V_k \gap \textrm{ with } \gap V_k = \mathbb{R}^{m_k}, \gap 1\leq k \leq N.
\end{equation*}
Let $\tilde{R}_k$:~$V \rightarrow V_k$ be the natural restriction operator, i.e.,
\begin{equation*}
u = [ \tilde{R}_k u ]_{k=1}^N := (\tilde{R}_1 u, \dots, \tilde{R}_N u), \quad u \in V.
\end{equation*}
We set $R_k^*$:~$V_k \rightarrow V$ to be the extension-by-zero operator.
Then we clearly have
\begin{equation*}
V = \sumk R_k^* V_k
\end{equation*}
and
\begin{equation*}
[u_k]_{k=1}^N = \sumk R_k^* u_k, \quad u_k \in V_k, \gap 1\leq k \leq N.
\end{equation*}
In addition, it is satisfied that
\begin{equation}
\label{RR}
\sumk R_k^* \tilde{R}_k = I.
\end{equation}
The following assumptions are imposed on $F$ and $G$.

\begin{assumption}
\label{Ass:Lip_block}
The functional $F$:~$V \rightarrow \mathbb{R}$ has the Lipschitz continuous derivative.
That is, there exists a constant $L > 0$ such that
\begin{equation*}
\| F'(u) - F'(v) \| \leq L \| u - v \|, \quad u,v \in V.
\end{equation*}
\end{assumption}

\begin{assumption}
\label{Ass:sep_block}
The functional $G$:~$V \rightarrow \mathbb{R}$ is block-separable.
That is, there exist functionals $G^k$:~$V_k \rightarrow \mathbb{R}$, $1 \leq k \leq N$, such that
\begin{equation*}
G \left( \left[ u_k \right]_{k=1}^N \right) = \sumk G^k ( u_k ).
\end{equation*}
\end{assumption}

In this setting, a simple parallel block coordinate descent method to solve~\cref{model_gradient} is presented in \cref{Alg:block}.

\begin{algorithm}[]
\caption{Parallel block coordinate descent method for~\cref{model_gradient}}
\begin{algorithmic}[]
\label{Alg:block}
\STATE Choose $u_k^{(0)} \in \dom G^k$, $1 \leq k \leq N$, and $\tau \in (0, 1/N ]$.
\FOR{$n=0,1,2,\dots$}
\item \vspace{-0.5cm} \begin{equation*}
\resizebox{0.9\textwidth}{!}{$ \displaystyle \begin{split}
v_k^{(n+1)} &= \argmin_{u_k \in V_k} \left\{ F(\un) + \langle F'(\un) , R_k^* (u_k - u_k^{(n)}) \rangle
+ \frac{L}{2} \| u_k - u_k^{(n)} \|^2 + G^k (u_k ) \right\}, \gap 1 \leq k \leq N \\
u_k^{(n+1)} &= (1-\tau )u_k^{(n)} + \tau v_k^{(n+1)}, \quad 1 \leq k \leq N
\end{split} $}
\end{equation*} \vspace{-0.4cm}
\ENDFOR
\end{algorithmic}
\end{algorithm}

In \cref{Alg:block}, let $\un = [u_k^{(n)}]_{k=1}^N$.
Then it is straightforward to observe that the sequence $\{ \un \}$ generated by \cref{Alg:block} is the same as the one generated by \cref{Alg:ASM} with $\tau_0  = 1/N$, $\omega = L$, and
\begin{align*}
d_k (w_k , v) &= \frac{1}{2} \| R_k^* w_k \|^2, \\
G_k(w_k, v) &= G^k ( \tilde{R}_k v + w_k ) + \sum_{j \neq k} G^j (\tilde{R}_j v)
\end{align*}
for $w_k \in V_k$, $v \in V$.

By using the Cauchy--Schwarz inequality and \cref{Ass:sep_block}, it is easy to check that \cref{Ass:stable} holds with $q= 2$ and $C_{0,K} = \sqrt{N}$ for all $K$.
Indeed, for any $u , v \in V$ with $u-v = [w_k]_{k=1}^N$ for $w_k \in V_k$, $1 \leq k \leq N$, it follows that
\begin{equation*}
\sumk d_k (w_k, v) = \frac{1}{2} \sumk \| R_k^* w_k \|^2 
\leq \frac{N}{2} \left\| \sumk R_k^* w_k \right\|^2 
= \frac{N}{2} \| u - v \|^2 
\end{equation*}
and
\begin{equation*} \begin{split}
\sumk G_k (w_k , v) &=  \sumk \left( G^k (\tilde{R}_k v + w_k ) + \sum_{j \neq k} G^j (\tilde{R}_j v) \right) \\
&= G \left( R_k^* (\tilde{R}_k v + w_k) \right) + (N-1) G \left( \sumk R_k^* \tilde{R}_k v_k \right) \\
&= G (u) + (N-1) G(v).
\end{split} \end{equation*}
\Cref{Ass:convex} is obvious.
\Cref{Ass:local} with $\omega_0 = L$ is a direct consequence of \cref{Ass:Lip_block}.
In conclusion, \cref{Ass:stable,Ass:convex,Ass:local} are verified and the $O(1/n)$ convergence of \cref{Alg:block} is obtained by \cref{Thm:ASM}.

Now, we turn our attention to constraint decomposition methods.
In~\cite{CTWY:2015,Tai:2003}, constraint decomposition methods were proposed as domain decomposition methods for nonlinear variational inequalities.
We show that those methods can be regarded as instances of \cref{Alg:ASM} with inexact local solvers.
In particular, we consider the one-level constraint decomposition method proposed in~\cite{Tai:2003} for~\cref{d_obstacle} with
\begin{equation*}
F(u) = \frac{1}{2} \intO |\nabla u|^2 \,dx - \left< f, u \right>;
\end{equation*}
the two-level method can be treated in a similar way.

We assume that the constraint  $K_h$ in~\cref{d_obstacle} is one-obstacle, i.e.,
\begin{equation*}
K_h = \left\{ u \in S_h (\Omega) : u \geq \underline{g} \right\}
\end{equation*}
for some $\underline{g} \in S_h (\Omega)$.
We also assume that the space $V = S_h (\Omega)$ is decomposed as~\cref{1L}.
We define operators $\tilde{R}_k$:~$V \rightarrow V_k$, $1 \leq k \leq N$ as
\begin{equation*}
\tilde{R}_k u = \left( I_h (\theta_k u) \right)|_{\Omega_k'}, \quad u \in V,
\end{equation*}
so  that~\cref{RR} holds.
If we set
\begin{equation*}
K_h^{k} = \left\{ u_k \in V_k : u_k \geq \tilde{R}_k \underline{g} \right\}, \quad 1 \leq k \leq N,
\end{equation*}
then it is clear that
\begin{equation*}
K_h = \sumk R_k^* K_h^k.
\end{equation*}
The constraint decomposition method proposed in~\cite{Tai:2003} in the above setting is summarized in \cref{Alg:constraint}.

\begin{algorithm}[]
\caption{Constraint decomposition method for~\cref{d_obstacle}}
\begin{algorithmic}[]
\label{Alg:constraint}
\STATE Choose $u^{(0)} \in K_h$ and $\tau \in (0, 1/N ]$.
\FOR{$n=0,1,2,\dots$}
\item \vspace{-0.5cm} \begin{eqnarray*}
v_k^{(n+1)} &\in& \argmin_{u_k \in V_k} \left\{ F \left( R_k^* u_k + \sum_{j \neq k} R_k^* \tilde{R}_k \un  \right) + \chi_{K_h^k} (u_k) \right\}, \quad 1 \leq k \leq N, \\
\unn &=& (1- \tau ) \un + \tau \sumk R_k^* v_k^{(n+1)}
\end{eqnarray*} \vspace{-0.4cm}
\ENDFOR
\end{algorithmic}
\end{algorithm}

One can check without major difficulty that \cref{Alg:constraint} is an instance of \cref{Alg:ASM} with $\tau_0 = 1/N$, $\omega = 1$, and
\begin{align*}
d_k (w_k, v) &= D_F (v + \Rw, v), \\
G_k (w_k, v) &= \chi_{K_h^k} (\tilde{R}_k v + w_k)
\end{align*}
for $w_k \in V_k$, $v \in V$.
In this sense, in order to prove the convergence of  \cref{Alg:constraint}, it suffices to verify \cref{Ass:stable,Ass:local}.

To verify \cref{Ass:stable}, for any $u, v \in \dom G$, we set $w_k = \tilde{R}_k u - \tilde{R}_k v$, $1 \leq k \leq N$.
Then we have
\begin{equation*}
u-v = \sumk \Rw.
\end{equation*}
Also, we get
\begin{equation*}
\sumk d_k (w_k , v ) \lesssim \left( 1 + \frac{1}{\delta^2} \right) \| u - v \|^2
\end{equation*}
by \cref{Lem:1L}, and 
\begin{equation*}
G_k (w_k , v) = \chi_{K_h^k} (\tilde{R}_k u) = 0, \quad 1 \leq k \leq N.
\end{equation*}
That is, \cref{Ass:stable} holds with $q = 2$ and
\begin{equation*}
C_{0,K} \lesssim 1 + \frac{1}{\delta}
\end{equation*}
for all K.

In \cref{Ass:local}, clearly we have $\omega_0 = 1$.
Moreover, we can prove
\begin{equation*}
\chi_{K_h} (v + \Rw) \leq \chi_{K_h^k} (\tilde{R}_k v  +w_k), \quad v \in \dom \chi_{K_h}, \gap w_k \in V_k,
\end{equation*}
as follows: for any interior node $x$ of $S_h (\Omega_k')$, we have
\begin{equation*} \begin{split}
(\tilde{R}_k v + w_k ) (x) \geq (\tilde{R}_k \underline{g})(x) \gap &\Leftrightarrow \gap \theta_k (x) v(x) + w_k (x) \geq \theta_k (x) \underline{g}(x) \\
&\Rightarrow \gap v(x) + w_k (x) \geq \underline{g}(x) \quad (\because v \in K_h ) \\
&\Leftrightarrow \gap (v + R_k^* w_k ) (x) \geq \underline{g}(x).
\end{split} \end{equation*}
Therefore, \cref{Ass:local} is proven.

Since the energy functional of~\cref{d_obstacle} satisfies \cref{Ass:sharp}, we conclude that \cref{Alg:constraint} converges linearly.
This result agrees with~\cite{Tai:2003}.

\section{Conclusion}
\label{Sec:Conclusion}
Motivated by the fact that additive Schwarz methods for linear elliptic problems can be represented as preconditioned Richardson methods, we showed that additive Schwarz methods for general convex optimization belong to a class of gradient methods.
From this observation, we presented a novel abstract convergence theory for additive Schwarz methods for convex optimization.
We also noted that the proposed theory directly generalizes the one presented in~\cite{TW:2005}, a standard framework for analyzing Schwarz methods for linear elliptic problems.
The proposed theory covers a fairly large range of convex optimization problems including constrained ones, nonsmooth ones, and nonsharp ones.
Moreover, the proposed theory is compatible with many existing works in the sense that it can adopt stable decomposition estimates from existing works with little modification.

There are several interesting topics for future research.
Due to the nonsymmetry of multiplicative Schwarz methods, they have no minimization structure like \cref{Lem:ASM}.
Since the proposed theory relies on the minimization structure of additive Schwarz methods, it is not applicable to multiplicative Schwarz methods.
Indeed,  in the field of mathematical optimization, analyzing multiplicative or alternating methods are considered to be much harder work than analyzing additive or parallel ones. 
Recently, the minimization structure of the symmetric block Gauss--Seidel method for quadratic programming was revealed in~\cite{LST:2019}.
We expect that a convergence theory for symmetric multiplicative Schwarz methods for general convex optimization can be designed by adopting the idea of~\cite{LST:2019}.

In the perspective of gradient methods, it is worth considering acceleration of additive Schwarz methods.
After a pioneering work of Nesterov~\cite{Nesterov:1983}, acceleration of gradient methods becomes a central topic in convex optimization.
In particular, an accelerated gradient method for the problem~\cref{model_gradient} was presented in~\cite{BT:2009}.
Recently, an accelerated block Jacobi method for a constrained quadratic optimization problem was proposed in~\cite{LP:2019b}.
Obtaining accelerated additive Schwarz methods for~\cref{model_gradient} by generalizing~\cite{LP:2019b} should be considered as a future work.

\appendix
\section{Technical proofs}
In this appendix, we provide missing proofs of lemmas and theorems that appeared in \cref{Sec:Pre,Sec:Gradient}.

\subsection{Proof of~\cref{Lem:infimal}}
\label{App:proof_infimal}
We note that the following proof only requires the vector space structure of spaces $W$ and $W_k$, $1 \leq k \leq N$; even the convexity of functionals $F_k$ is not assumed.

\begin{proof}[Proof of~\cref{Lem:infimal}]
In this proof, an index $k$ runs from $1$ to $N$.
Take any $w \in W$.
For $w_k \in W_k$ satisfying $w = \sumk A_k w_k$, by the definitions of infimal postcomposition and infimal convolution, we have
\begin{equation*}
\sumk  F_k (w_k) \geq \sumk (A_k \triangleright F_k) (A_k w_k)
\geq \left( \bigsquare_{k=1}^N (A_k \triangleright F_k) \right) (w).
\end{equation*}
Taking the infimum in the left-hand side of the above equation yields
\begin{equation*}
\left( \bigsquare_{k=1}^N (A_k \triangleright F_k) \right) (w) \leq \inf \left\{ \sumk F_k (w_k) : w = \sumk A_k w_k, \gap w_k \in W_k \right\} .
\end{equation*}

Now, we show the reverse direction.
For convenience, we write
\begin{equation*}
\Delta = \left( \bigsquare_{k=1}^N (A_k \triangleright F_k) \right) (w).
\end{equation*}
If $\Delta = \infty$, there is nothing to show.
For the case $\Delta \in \mathbb{R}$, choose any $\epsilon > 0$.
Then there exists $w^k \in V$ with $w = \sumk w^k$ such that
\begin{equation}
\label{infimal1}
\Delta \leq \sumk (A_k \triangleright F_k) (w^k) \leq \Delta + \frac{\epsilon}{2}.
\end{equation}
Thus $(A_k \triangleright F_k) (w^k)$ is finite for every $k$, and there exists $\bar{w}_k \in V_k$ with $w^k = A_k \bar{w}_k$ such that
\begin{equation}
\label{infimal2}
F_k (\bar{w}_k ) \leq (A_k \triangleright F_k ) (w^k ) + \frac{\epsilon}{2N}.
\end{equation}
Summation of~\cref{infimal1,infimal2} over all $k$ yields
\begin{equation*}
\sumk F_k (\bar{w}_k) \leq \sumk (A_k \triangleright F_k ) (w^k) + \frac{\epsilon}{2} \leq \Delta + \epsilon.
\end{equation*}
Since $w = \sumk w^k = \sumk A_k \bar{w}_k$ and $\epsilon$ was chosen arbitrary, we get
\begin{equation*}
\inf \left\{ \sumk F_k (w_k) : w = \sumk A_k w_k, \gap w_k \in W_k \right\} \leq \Delta.
\end{equation*}
Finally, we consider the case $\Delta = -\infty$.
Take any $M > 0$.
One can choose $w^k \in V$ with $w = \sumk w^k$ such that
\begin{equation}
\label{infimal3}
\sumk (A_k \triangleright F_k ) (w^k) \leq - 2M.
\end{equation}
We define the following two index sets $\mathcal{I}_1$ and $\mathcal{I}_2$ as follows:
\begin{equation*} \begin{split}
\mathcal{I}_1 &= \left\{ k : (A_k \triangleright F_k ) (w^k) \in \mathbb{R} \right\}, \\
\mathcal{I}_2 &= \left\{ k : (A_k \triangleright F_k ) (w^k) = -\infty \right\}. \\
\end{split} \end{equation*}
Clearly, $\mathcal{I}_1 \cup \mathcal{I}_2 = \left\{1, \dots, N \right\}$.
If $\mathcal{I}_2 = \emptyset$, there exist $\bar{w}_k \in V_k$ for all $k$ satisfying $w^k = A_k \bar{w}_k$ such that
\begin{equation}
\label{infimal4}
F_k (\bar{w}_k ) \leq (A_k \triangleright F_k ) (w^k) + \frac{M}{N}.
\end{equation}
Combining with~\cref{infimal3} followed by summing~\cref{infimal4} over all $k$ yields
\begin{equation}
\label{infimal5}
\sumk F_k (\bar{w}_k) \leq \sumk (A_k \triangleright F_k) (w^k) + M \leq -M.
\end{equation}
If $\mathcal{I}_2 \neq \emptyset$, one may choose $\bar{w}_k$ with $w^k = A_k \bar{w}_k$ such that
\begin{align*}
F_j (\bar{w}_j ) &\leq (A_j \triangleright F_j ) (w^j) + \frac{M}{N}, \hspace{2.5cm} j \in \mathcal{I}_1, \\
F_j (\bar{w}_j ) &\leq -\frac{1}{|\mathcal{I}_2|} \left( \sum_{i \in \mathcal{I}_1} (A_i \triangleright F_i ) (w^i) + 2M \right), \quad j \in \mathcal{I}_2.
\end{align*}
Summing $F_k (\bar{w}_k)$ over all $k$ yields
\begin{equation} 
\label{infimal6} \begin{split}
\sumk F_k (\bar{w}_k ) &= \sum_{j \in \mathcal{I}_1} F_j (\bar{w}_j) + \sum_{j \in \mathcal{I}_2} F_j (\bar{w}_j) \\
&\leq \left( \sum_{j \in \mathcal{I}_1} (A_j \triangleright F_j ) (w^j) + \frac{|\mathcal{I}_1| M}{N} \right) - \left( \sum_{j \in \mathcal{I}_1} (A_j \triangleright F_j ) (w^j) + 2M \right) \\
&= \left(-2 + \frac{|\mathcal{I}_1|}{N }\right) M \\
&< -M.
\end{split}\end{equation}
In both cases~\cref{infimal5,infimal6}, we conclude that
\begin{equation*}
\inf \left\{ \sumk F_k (w_k) : w = \sumk A_k w_k, \gap w_k \in W_k \right\} = -\infty ,
\end{equation*}
as $M$ can be arbitrarily large.
\end{proof}

\subsection{Proof of~\cref{Lem:decreasing}}
\label{App:proof_decreasing}
Take any $n \geq 0$.
It is obvious that there exists a bounded and convex subset $K$ of $V$ such that $\un, \unn \in K$.
By \cref{Ass:gradient} and the minimization property of $\unn$, we have
\begin{equation*} \begin{split}
E(\unn) &= F(\un) + \langle F'(\un), \unn - \un \rangle + D_F (\unn, \un ) + G(\unn) \\
&\leq Q (\unn, \un) \\
&\leq Q (\un, \un) \\
&= F(\un) + B(\un , \un ) \\
&\leq F(\un) + G(\un).
\end{split} \end{equation*}
Therefore, we conclude that $E(\unn) \leq E(\un)$.

\subsection{Proof of \cref{Thm:gradient}}
\label{App:proof_gradient}
In order to estimate the convergence rate of \cref{Alg:gradient}, we need the following useful lemmas.

\begin{lemma}
\label{Lem:recur}
Let $\left\{ a_n \right\}$ be a sequence of positive real numbers which satisfies
\begin{equation*}
a_n - a_{n+1} \geq C a_n^{\gamma}, \quad n \geq 0,
\end{equation*}
for some $C > 0$ and $\gamma > 1$.
Then with $\beta = \frac{1}{\gamma - 1}$, we have
\begin{equation*}
a_n \leq \frac{1}{(n+1)^{\beta}} \max \left\{ a_0, \left( \frac{2^{\beta} - 1}{C} \right)^{\beta} \right\}, \quad n \geq 0.
\end{equation*}
\end{lemma}
\begin{proof}
See~\cite[Lemma~1]{HLL:2007}.
\end{proof}

\begin{lemma}
\label{Lem:q_min}
Let $a,b > 0$, $q > 1$, and $\theta \in (0, 1]$.
The minimum of the function $g(t) = \frac{a}{q}t^q - bt$, $t \in [0, \theta]$, is given as follows:
\begin{equation*}
\mint g(t) = \begin{cases} \frac{a}{q}\theta^q - b\theta \leq -b\theta \left( 1 - \frac{1}{q} \right) & \textrm{ if }\gap a \theta^{q-1} - b \leq 0, \\
-b \left( \frac{b}{a} \right)^{\frac{1}{q-1}} \left( 1 - \frac{1}{q} \right) & \textrm{ if }\gap  a \theta^{q-1} - b > 0. \end{cases}
\end{equation*}
\end{lemma}
\begin{proof}
It is elementary.
\end{proof}

We notice that the proof of \cref{Thm:gradient} is motivated by~\cite[Theorem~4]{Nesterov:2013}, where a special case, forward-backward splitting with $q=2$ and $\theta = 1$, was analyzed.

\begin{proof}[Proof of~\cref{Thm:gradient}]
Take any $n \geq 0$.
For $u \in K_0$, we write
\begin{equation}
\label{utheta}
u_{\theta} = \frac{1}{\theta} u - \left( \frac{1}{\theta} - 1 \right) \un ,
\end{equation}
so that $u - \un = \theta ( u_{\theta} - \un )$.
Note that if we set $u = tu^* + (1 - t ) \un$ for $t \in [0, \theta]$, then $u \in K_0$ and
\begin{equation}
\label{utheta2}
u_{\theta} = \frac{t}{\theta} u^* + \left( 1 - \frac{t}{\theta} \right) \un \in K_0.
\end{equation} 
We denote $E(\un) - E(u^*)$ by $\zeta_n$.
It follows that
\begin{equation}
\label{gradient1} 
\resizebox{\textwidth}{!}{$ \displaystyle \begin{split}
E(&\unn) = F(\un) + \langle F'(\un) , \unn - \un \rangle + D_F (\unn, \un) + G(\unn) \\
&\stackrel{\textrm{(a)}}{\leq}  Q_n (\unn) \\
&= \min_{u \in K_0} \left\{ F(\un) + \langle F'(\un) , u - \un \rangle + B (u, \un) \right\} \\
&\stackrel{\textrm{(b)}, \cref{utheta}}{\leq} \min_{u \in K_0} \left\{ F(\un) + \theta \langle F'(\un), u_{\theta} - \un \rangle + \frac{L\theta^q}{q} \| u_{\theta} - \un \|^q  + \theta G \left( u_{\theta} \right) + (1- \theta) G(\un) \right\} \\
&\stackrel{\textrm{(c)}}{\leq} \min_{u \in K_0} \left\{ (1-\theta)E(\un) + \theta E(u_{\theta}) + \frac{L \theta^q}{q} \| u_{\theta} - \un \|^q \right\} \\
&\stackrel{\cref{utheta2}}\leq \mint \left\{ (1-\theta) E(\un ) + \theta E \left( \frac{t}{\theta} u^* + \left(1-\frac{t}{\theta} \right)\un \right) + \frac{Lt^q}{q} \| u^* - \un \|^q  \right\} \\
&\stackrel{\textrm{(d)}}{\leq} \mint \left\{ E(\un) - t\zeta_n + \frac{Lt^q}{q} R_0^q \right\},
\end{split} $} \end{equation}
where (a), (b) are because of \cref{Ass:gradient}, (c) is due to the convexity of $F$, and (d) is due to the convexity of $E$.
Invoking \cref{Lem:q_min}, we have
\begin{equation*}
E(\unn) \leq E(\un) - \theta \left( 1 - \frac{1}{q} \right) \zeta_n
\end{equation*}
if $\zeta_n \geq \theta^{q-1}LR_0^q$, which is equivalent to
\begin{equation*}
\zeta_{n+1} \leq \left( 1 - \theta \left( 1 - \frac{1}{q} \right)\right) \zeta_n.
\end{equation*}
Otherwise, if $\zeta_n < \theta^{q-1} LR_0^q$, we get
\begin{equation*}
E(\unn) \leq E(\un) - \frac{1}{(LR_0^q)^{\frac{1}{q-1}}}\left( 1- \frac{1}{q}  \right) \zeta_n^{\frac{q}{q-1}}.
\end{equation*}
Hence, we have
\begin{equation*}
\zeta_n - \zeta_{n+1} \geq \frac{1}{(LR_0^q)^{\frac{1}{q-1}}}\left( 1- \frac{1}{q}  \right) \zeta_n^{\frac{q}{q-1}}.
\end{equation*}
Invoking \cref{Lem:recur} yields
\begin{equation*}
\zeta_{n} \leq \frac{1}{(n+1)^{q-1}} \max \left\{ \zeta_0 , \left( \frac{q(2^{q-1}-1)}{q-1} \right)^{q-1} L R_0^q \right\}, \quad n \geq 0.
\end{equation*}
Since $\zeta_0 < \theta^{q-1} LR_0^q$, setting
\begin{equation}
\label{Cq}
C_{q,\theta} = \left( \max \left\{ \theta, \left( \frac{q(2^{q-1}-1)}{q-1} \right) \right\} \right)^{q-1}
\end{equation}
completes the proof.
\end{proof}

\subsection{Proof of \cref{Thm:gradient_uniform}}
\label{App:proof_gradient_uniform}
The proof of \cref{Thm:gradient_uniform} is done with a similar argument to~\cite[Theorem~5]{Nesterov:2013}, where the convergence analysis for $p=q=2$ was given.

\begin{proof}[Proof of~\cref{Thm:gradient_uniform}]
Again, we denote $E(u^{(n)}) - E(u^*)$ by $\zeta_n$.
By the same way as~\cref{gradient1}, one can obtain
\begin{equation}
\label{gradient_uniform1}
E(u^{(n+1)}) \leq \mint \left\{E(\un) - t \zeta_n + \frac{Lt^q}{q} \| u^* - \un \|^q \right\}.
\end{equation}
By~\cref{Ass:sharp}, we have
\begin{equation}
\label{gradient_uniform2} 
\| u^* - \un \|^q \leq \left( \frac{p}{\mu} \zeta_n \right)^{\frac{q}{p}}.
 \end{equation}
Combining~\cref{gradient_uniform1,gradient_uniform2} yields
\begin{equation}
\label{gradient_uniform3}
E(\unn) \leq \mint \left\{ E(\un) - t \zeta_n + \frac{p^{\frac{q}{p}}Lt^q}{q\mu^{\frac{q}{p}}} \zeta_n^{\frac{q}{p}} \right\}.
\end{equation}

We consider the two cases $p = q$ and $p > q$ separately.
First, we assume that $p = q$.
Then~\cref{gradient_uniform3} is simplified to
\begin{equation*}
\zeta_{n+1} \leq \mint \left(1-t + \frac{Lt^q}{\mu} \right) \zeta_n.
\end{equation*}
By \cref{Lem:q_min}, if $\mu \geq \theta^{q-1}qL$, then we get
\begin{equation}
\label{gradient_uniform4}
\zeta_{n+1} \leq \left( 1- \theta \left( 1 - \frac{1}{q} \right) \right) \zeta_n.
\end{equation}
Otherwise, if $\mu < \theta^{q-1}qL$, we have
\begin{equation}
\label{gradient_uniform5}
\zeta_{n+1} \leq \left( 1 - \left( 1 - \frac{1}{q} \right) \left( \frac{\mu}{qL} \right)^{\frac{1}{q-1}} \right) \zeta_n.
\end{equation}
Recursive applications of~\cref{gradient_uniform4,gradient_uniform5} yield the desired results for the case $p = q$.

Next, we consider the case $p > q$.
Let $r = \frac{q}{p} < 1$.
By \cref{Lem:q_min}, if
\begin{equation*}
\zeta_n^{1-r} \geq \frac{\theta^{q-1} p^r L}{\mu^r},
\end{equation*}
one can obtain from~\cref{gradient_uniform3} that
\begin{equation*}
\zeta_{n+1} \leq \left( 1- \theta \left( 1- \frac{1}{q} \right)\right) \zeta_n.
\end{equation*}
Otherwise, it follows that
\begin{equation*}
\zeta_{n+1} \leq \zeta_n - \left( 1- \frac{1}{q} \right) \left( \frac{\mu^{r}}{p^r L}\right)^{\frac{1}{q-1}} \zeta_n^{\frac{q(p-1)}{p(q-1)}}.
\end{equation*}
Application of \cref{Lem:recur} yields
\begin{equation*}
\zeta_n \leq \frac{1}{(n+1)^{\beta}} \max \left\{ \zeta_0 , \left( \frac{q (2^{\beta} - 1)}{q-1} \right)^{\beta} \left( \frac{p^r L}{\mu^r} \right)^{\frac{1}{1-r}} \right\},
\end{equation*}
where $\beta = \frac{p(q-1)}{p-q}$.
Since $\zeta_0 \leq \theta^{\frac{q-1}{1-r}}\left( \frac{p^r L}{\mu^r}\right)^{\frac{1}{1-r}}$, we conclude that
\begin{equation*}
\zeta_n \leq \frac{C_{p,q,\theta}(L/\mu^r )^{\frac{1}{1-r}}}{(n+1)^{\beta}}
\end{equation*}
with
\begin{equation}
\label{Cpq}
C_{p,q,\theta} = p^{\frac{q}{p-q}}\left( \max \left\{ \theta, \left( \frac{q (2^{\frac{p(q-1)}{p-q}} - 1)}{q-1} \right) \right\} \right)^{\frac{p(q-1)}{p-q}}.
\end{equation}
This completes the proof.
\end{proof}

\bibliographystyle{siamplain}
\bibliography{refs_ASM_GD}
\end{document}